\documentclass[12pt,oneside,english]{amsart}
\usepackage{}
\usepackage{amstext}
\usepackage{amsthm}
\usepackage[T1]{fontenc}
\usepackage[utf8]{inputenc}
\usepackage{geometry}
\usepackage{mathrsfs}
\usepackage{graphicx,epstopdf}
\usepackage[colorlinks,bookmarks,linkcolor=black,citecolor=black]{hyperref} 
\usepackage{enumitem}
\usepackage[english,algoruled,lined,noresetcount,norelsize]{algorithm2e}

\makeatletter
\numberwithin{equation}{section}
\numberwithin{figure}{section}

\newcommand{\dist}{\mathrm{dist}}

\newcommand{\avdeg}{\mathrm{d\overline{eg}}}
\newcommand{\dens}{d}
\newcommand{\NN}{\mathbb N}

\newcommand{\just}[1]{\mbox{\fbox{\tiny#1}}\quad}

\newcommand{\fG}{\mathcal{G}}
\newcommand{\fF}{\mathcal{F}}
\newcommand{\fD}{\mathcal{D}}
\newcommand{\fH}{\mathcal{H}}

\newcommand{\fB}{\mathcal{B}}

\newcommand{\fS}{\mathcal{S}}
\newcommand{\fT}{\mathcal{T}}
\newcommand{\fL}{\mathcal{L}}
\newcommand{\tr}{{\tilde r}}

\newcommand{\biga}{\tr a_1 > (1-\tr) a_2}
\newcommand{\bigb}{\tr a_1 \le (1-\tr) a_2}
\newcommand{\aone}{|\fD_{A1}|}
\newcommand{\atwo}{|\fD_{A2}|}
\newcommand{\bone}{|\fD_{B2}|}

\newcommand{\psm}{{P\ref{prop:embed-SM}}}
\newcommand{\pls}{{P\ref{prop:embed-LS}}}

\newtheorem{thm}{Theorem}
\newtheorem{theorem}{Theorem}

\newtheorem{lemma}[thm]{Lemma}
\newtheorem{claim}[thm]{Claim}
\newtheorem{proposition}[thm]{Proposition}
\newtheorem{corollary}[thm]{Corollary}

\newtheorem{conj}[thm]{Conjecture}

\theoremstyle{definition}

\newtheorem{definition}[thm]{Definition}

\theoremstyle{remark}

\renewcommand{\phi}{\varphi}
\renewcommand{\marginpar}[1]{} %sunda globalne vsechni marginpar

\title{A version of the Loebl-Koml\'os-S\'os conjecture for skewed trees}

\author{Tereza Klimo\v sov\'a}\address{Department of Applied Mathematics, Faculty of Mathematics and Physics, Charles University, 
	Malostransk\'e n\'am\v{e}st\'i 25, 118 00 Praha 1, Czech Republic.}
\email{tereza@kam.mff.cuni.cz}
\author{Diana Piguet}\address{The Czech Academy of Sciences, Institute of Computer Science, Pod Vod\'{a}renskou v\v{e}\v{z}\'{\i} 2, 182 07 Prague, Czech Republic. With institutional support RVO:67985807.}
\email{piguet@cs.cas.cz}
\author{V\'aclav Rozho\v n}
\address{Faculty of Mathematics and Physics, Charles University \& The Czech Academy of Sciences, Institute of Computer Science, Pod Vod\'{a}renskou v\v{e}\v{z}\'{\i} 2, 182 07 Prague, Czech Republic. With institutional support RVO:67985807.}\email{vaclavrozhon@gmail.com}
\thanks{Klimo\v sov\'a was supported by Center
of Excellence -- ITI, project P202/12/G061 of GA \v{C}R, Piguet and Rozho\v n were supported by the Czech Science Foundation, grant number GJ16-07822Y. \\
Extended abstract of this work was published as~\cite{KLIMOSOVA2017}.}
\begin{document}

\maketitle
\begin{abstract}
     Loebl, Koml\'os, and S\'os conjectured that any graph with at least half of its vertices of degree at least~$k$ contains every tree with at most~$k$ edges. We propose a version of this conjecture for \emph{skewed} trees, i.e., we consider the class of trees with at most~$k$ edges such that the sizes of the colour classes of the trees have a given ratio. We show that our conjecture is asymptotically correct for dense graphs. The proof relies on the regularity method. Our result implies bounds on Ramsey number of several trees of given skew. 
\end{abstract}

\section{Introduction}
%so far just copied from EUROCOMB, edited

Many problems in extremal graph theory ask whether a certain density condition imposed on a host graph guarantees the containment of a given subgraph~$H$. Typically, the density condition is expressed by average or minimum degree. Classical examples of results of this type are Tur\'{a}n's Theorem which determines the average degree that guarantees the containment of the complete graph $K_r$ and the Erd\H os-Stone Theorem~\cite{Erdos1946} which  essentially determines the average degree condition guaranteeing the containment of a fixed non-bipartite graph~$H$.
 On the other hand, for a general bipartite graph~$H$ the problem is wide open. For $H$ being a tree, the long-standing conjecture of Erd\H os and S\'os from 1962 asserts that an average degree greater than $k-1$ forces a copy of any tree of order $k+1$. (Note that a trivial bound for the average degree guaranteeing containment of such a tree is~$2k$, since in such a graph we can find a subgraph of minimum degree at least~$k$ and then embed~the tree greedily.) A solution of this conjecture for large~$k$, based on an extension of the Regularity Lemma, has been announced in the early 1990's by Ajtai, Koml\'os, Simonovits, and Szemer\'edi~\cite{Ajtai}.

%A different approach to the problem is to relax the condition of minimum degree by investigating how many vertices of degree~$k$ guarantee the containment of a tree of order~$k+1$. 

The problem of containing a tree of a given size has also been studied in settings with different density requirements. Recently, Havet, Reed, Stein and Wood conjectured that a graph of minimum degree at least $\lfloor\frac{2k}{3}\rfloor$ and maximum degree at least $k$ contains a copy of any tree of order $k+1$ and provided some evidence for this conjecture.%k+1 vertices

Another type of density requirement, on which we focus in this paper, is considered in the Loebl-Koml\'os-S\'os conjecture. The conjecture asserts that at least half of the vertices of degree at least~$k$ guarantees containment of a tree of order $k+1$. In other words, the requirement of average degree in Erd\H os-S\'os conjecture is replaced by a median degree condition. 

The conjecture has been solved exactly for large dense graphs \cite{Cooley2009,Hladkyn} and proved to be asymptotically true for sparse graphs \cite{HladkyLKS1,HladkyLKS2,HladkyLKS3,HladkyLKS4} (see~\cite{Hladky2015} for an overview). %tuhle cast jsem nemenila

All these conjectures are known to be best possible.
In particular, the Loebl-Koml\'os-S\'os conjecture is tight for paths. %odsud dal nezmenene
To observe this, consider a graph consisting of a disjoint union of copies of a graph~$H$ of order~$k+1$ consisting of a clique of size $\lfloor \frac {k+1}2\rfloor-1$, an independent set on the remaining vertices, and the complete bipartite graph between the two sets. Almost half of the vertices of this graph have degree~$k$, but it does not contain a path on~$k+1$ vertices as a subgraph.

A natural question is whether fewer vertices of degree~$k$ suffice when one considers only a restricted class of trees. Specifically, Simonovits asked [personal communication], whether it is the case for trees of given {\em skew}, that is, the ratio of sizes of the smaller and the larger colour classes is bounded by a constant smaller than~$1$.
We propose the following conjecture.

\begin{conj}\label{conj}
Any graph of order~$n$ with at least~$rn$ vertices of degree at least~$k$ contains every tree of order at most~$k+1$ with colour classes $V_1, V_2$ such that $|V_1|\le r \cdot (k+1)$. 
\end{conj}

%TODO: will this be a section or a separate paper? We have verified that the conjecture holds both for trees of diameter at most five and for paths on at most $2r(k+1)$ vertices.

 If true, our conjecture is best possible for the similar reason as the Loebl-Koml\'os-S\'os conjecture. Indeed, given $r\in (0,1/2]$, consider a graph consisting of a disjoint union of copies of a graph~$H$ with $k+1$ vertices consisting of a clique of size $\lfloor r(k+1)\rfloor-1$, an independent set on the remaining vertices and the complete bipartite graph between the two sets (see Figure \ref{img:excase}). Such a graph does not contain a path on $2\lfloor r(k+1)\rfloor$ vertices (or, to give an example of a tree of maximal order, a path on $2\lfloor r(k+1)\rfloor$ vertices with one end-vertex identified with the centre of a star with $k+1-2\lfloor r(k+1)\rfloor$ leaves).

\begin{figure}
	\centering
	\includegraphics[width=.8\textwidth]{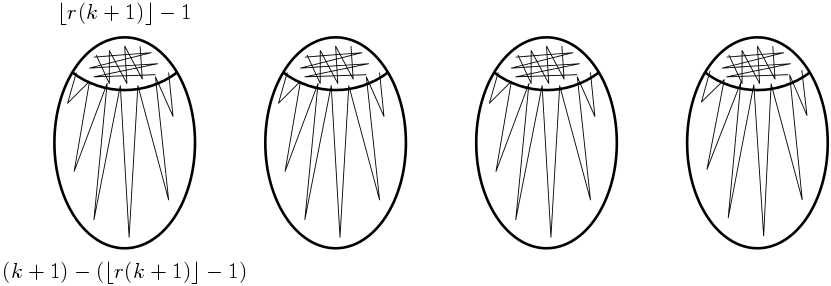}
	\caption{The graph showing the tightness of Conjecture~\ref{conj} is a disjoint union of graphs of order $k+1$.}
	\label{img:excase}
\end{figure}

We verified that the conjecture is true both for paths and for trees of diameter at most five \cite{Rozhon2018}. %\marginpar{V: pridal jsem odkaz na svou bakalarku}
In this paper we prove that Conjecture~\ref{conj} is asymptotically correct for dense graphs. 
\begin{theorem}\label{thm:result}
Let $0<r\le 1/2$ and $q>0$. Then for any $\eta>0$ there exists $n_0\in \mathbb N$ such that for every $n\ge n_0$ and $k\geq qn$, any graph of order~$n$ with at least $rn$ vertices of degree at least $(1+\eta)k$ contains every tree of order at most~$k$ with colour classes $V_1, V_2$ such that $|V_1|\le rk$. 
\end{theorem}
\marginpar{T: pridala jsem par vet o dukazu, budu rada kdyz si prectete, ze to dava smysl V: mne to dava smysl, D: mne to smysl moc nedavalo. Tvuj text jsem zachovala v komentare a skusila jsem jinou verzi. Dejte vedet, jestli tu vam take da smysl. V: za me OK}
This extends the main result of~\cite{PS07+}, which is a special case of Theorem~\ref{thm:result} for $r=1/2$.
 While we use and extend some of their techniques, our analysis is more complex. 
 As in~\cite{PS07+}, we partition the tree into small rooted subtrees, which we then embed into regular pairs of the host graph. In order to connect those small rooted trees, we need two adjacent clusters with adquate average degree to those regular pairs, which typically will be represented by a matching in the cluster graph. Hence, we need a matching in the cluster graph that is as large as possible. For this aim, we use disbalanced regularity decomposition (see~\cite{Haxell2002}), placing large degree vertices into smaller clusters than the remaining vertices, hence covering as many low degree vertices as possible by this matching.
 We then consider several possible embedding configurations in the regularity decomposition, depending on the structure of the cluster graph, in particular depending on the properties of the adjacent clusters with suitable average degree to the optimal matching. 
%  As in~\cite{PS07+}, we partition the tree into small rooted subtrees, which we then embed into the host graph based on the structure of its the regular decomposition. Based on  properties of this tree partition, we consider several embedding configurations in the regularity decomposition based on matching between regular pairs.
% To make full use of large degree vertices, we use disbalanced regularity decomposition (see~\cite{Haxell2002}), placing large degree vertices into smaller clusters than the remaining vertices.
%Our result is a skew analogue of the result of In particular, their result is a special case of ours for $r=1/2$. 
%To compensate for the small number of large degree vertices, we use a disbalanced regularity decomposition (see~\cite{Haxell2002}), placing large degree vertices into smaller clusters than the remaining vertices.
%In this way, we obtain a cluster graph where at least half of the clusters consist of large degree vertices. 
%
%To obtain one of several suitable embedding configurations, we exploit a matching in the cluster graph that minimizes the number of uncovered clusters containing small degree vertices but, unlike in~\cite{PS07+}, we do not need to use Edmonds-Gallai Theorem to find such a matching. 
%Also, unlike in~\cite{PS07+}, the suitable embedding configurations will differ depending on the rough structure of the tree, in particular, how the partition into smaller trees splits the colour classes of the tree.

The structure of the rest of the paper is the following; in Section~\ref{sec:reg}, we introduce notions and results related to regularity. In Section~\ref{sec:prelim} we introduce tools necessary for the proof of Theorem~\ref{thm:result} which we present in Section~\ref{sec:proof}. In Sections~\ref{sec:comb} , ~\ref{sec:embedding} and~\ref{sec:main_prop} we prove our main tools from Section~\ref{sec:prelim};  Propositions~\ref{prop:comb} and~\ref{prop:mainembedding2}. 
In Section~\ref{sec:conclusion} we discuss implications of our conjecture for Ramsey numbers of trees and further research directions.

%The reader interested in the main ideas of the proof can consult the extended abstract of this paper \cite{KLIMOSOVA2017} containing its summary. \marginpar{V: dal jsem odkaz na ex. abstrakt a prijde nadbytecne psat summary dukazu do tohoto paperu}
\label{sec:intro}

\section{Regularity}\label{sec:reg}
In this section we introduce a notion of {\em regular pair}, state the regularity lemma and introduce a standard method of embedding a tree into a regular pair.

  Let $G$ be a graph and let $X,Y$ be disjoint subsets of its vertices and $\varepsilon>0$. We define $E(X,Y)$ as the set of edges of $G$ with one end in $X$ and one end in $Y$ and the {\em density} of  the pair $(X,Y)$ as $\dens(X,Y)=\frac{|E(X,Y)|}{|X||Y|}$. The \emph{degree} $\deg(x)$ of a vertex~$x$ is the number of its neighbours. By  $\deg(x,X)$ we denote the number of neighbours of~$x$ in the set~$X$.
 We say that $(X,Y)$ is an {\em $\varepsilon$-regular pair}, if for every $X'\subseteq X$ and $Y'\subseteq Y$, $|X'|\geq \varepsilon |X|$ and $|Y'|\geq \varepsilon |Y|$, $\dens(X',Y')-\dens(X,Y)\leq \varepsilon$.

The following lemma states a well-known fact that subsets of a regular pair to some extent 'inherit' regularity of the whole pair.

\begin{lemma}\label{lem:subdivide}
 Let $G$ be a graph and $(X,Y)$ be an $\varepsilon$-regular pair of density $d$ in $G$. Let $X'\subseteq X$ and $Y\subseteq Y$ such that $|X'|\geq \alpha |X|$ and $|Y'|\geq \alpha |Y|$. Then, $(X',Y')$ is an $\varepsilon'$-regular pair of density at least $d-\varepsilon$, where $\varepsilon'=\max(\varepsilon/\alpha,2\varepsilon)$.
%v dense LKS podminka alpha>2eps, v [viz nize] ze ktereho cituji ne
%fact 1.5
%J. Komlós, M. Simonovits
%Szemerédi's regularity lemma and its applications in graph theory
%Combinatorics, Paul Erdős is Eighty, vol. 2, Bolyai Soc. Math. Stud., vol. 2, Keszthely, 1993, János Bolyai Math. Soc., Budapest (1996), pp. 295-352
\end{lemma}

We say that a partition $\{V_0,V_1,\ldots, V_N\}$ of $V(G)$ is an {\em $\varepsilon$-regular partition}, if $|V_0|\leq \varepsilon |V(G)|$ all but at most $\varepsilon N^2$ pairs $(V_i,V_j)$, $i<j$, $i,j\in[N]$, are $\varepsilon$-regular. We call $V_0$ a {\em garbage set}. We call a regular partition {\em equitable} if $|V_i|=|V_j|$ for every $i,j\in [N]$.

\begin{lemma}[Szem\'eredi regularity lemma] \label{lem:reg}
For every $\varepsilon>0$ and $N_{min}\in \NN$ there exists $N_{max} \in \NN$ and $n_R\in \NN$ such that every graph $G$ on at least $n_R$ vertices admits an $\varepsilon$-regular partition $\{V_0,\ldots, V_N\}$, where $N_{min}\leq N \leq N_{max}$.
\end{lemma}

Given an $\varepsilon$-regular pair $(X,Y)$,
%, we say that a subset of $X$ or $Y$ is {\em significant} if has size at least $\varepsilon|X|$ or $\varepilon|Y|$, respectively. 
we call a vertex $x\in X$ {\em typical} with respect to a set $Y'\subseteq Y$ if $\deg(x,Y')\ge  (d(X,Y)-\varepsilon)|Y'|$.
  Note that from the definition of regularity it follows that all but at most $\varepsilon |X|$ vertices of $X$ are typical with respect to any subset of $Y$ of size at least $\varepsilon|Y|$. This observation can be strengthened as follows.

\begin{lemma}[Variant of Proposition 4.5 in~\cite{Zhao2011}] \label{lem:ultratypical}
Let $\{V_0,V_1,\ldots, V_N\}$ be an $\varepsilon$-regular partition of $V(G)$ and let $X=V_j$ for some $j\in [N]$.
Then all but at most $\sqrt{\varepsilon}|X|$ vertices of a cluster $X$ are typical w. r. t. all but at most $\sqrt{\varepsilon}N$ sets $V_i$, $i\in [N]\setminus {j}$. We call such vertices of $X$ {\em ultratypical}.
\end{lemma}

\begin{proof}
Suppose, for a contradiction, that there are more than $\sqrt{\epsilon} |X|$ vertices of a cluster $X$ that are not typical to more than $\sqrt{\epsilon} N$ clusters. Then we have at least $\sqrt{\epsilon}|X| \cdot \sqrt{\epsilon}N = \epsilon |X|N$ pairs formed by a cluster and a vertex from $X$ not typical to that cluster. This in turn means that there is a cluster $Y$ such that the number of vertices not typical to $Y$ is at least $\epsilon |X|$. But then the set of these vertices contradicts the regularity of the pair $XY$. 
\end{proof}

%Next, we state standard two lemmas for embedding trees into a regular pair. 
Next lemma states that a tree can be embedded in a sufficiently large subset of a regular pair, each of the colour classes being embedded in one 'side'. Moreover we can prescribe embedding of a few vertices.
%The second lemma then asserts that one can embed a collection of many rooted trees in a similar manner.
%\marginpar{T: zmenila jsem 4 na 2 v $|X'| > 2\frac{\varepsilon}{\alpha}|X|$, pokud jsem nekde neco nezvrtala, melo by to fungovat stejne dobre V: prijde mi, ze jo}
\begin{lemma}\label{lem:tree-emb}
Let $T$ be a tree with colour classes $F_1$ and $F_2$. Let $R\subseteq F_1$, $|R|\leq 2$ such that vertices of $R$ do not have a common neighbour in $T$ (if $|R|=2$). 

Let $\varepsilon > 0$ and $\alpha > 2\varepsilon$. Let $(X, Y)$ be an $\varepsilon$-regular pair in a graph $G$ with density $\dens(X, Y) > 3\alpha$ such that $|F_1|\leq \varepsilon |X|$ and $|F_2|\leq \varepsilon |Y|$. Let  $X' \subseteq X, Y' \subseteq Y$ be sets satisfying $|X'| > 2\frac{\varepsilon}{\alpha}|X|, |Y'| > 2\frac{\varepsilon}{\alpha}|Y|$.

Let $\varphi$ be any injective mapping of vertices of $R$ to vertices of $X'$ with degree greater than $3\varepsilon|Y|$ in $Y'$. Then there exists extension of $\varphi$ that is an injective homomorphism from $T$ to $(X,Y)$ satisfying $\varphi(F_1)\subseteq X'$ and $\varphi(F_2)\subseteq Y'$.

%Lemma 8.5 from https://arxiv.org/pdf/1211.3050.pdf%
%chceme ruzne velke C,D
%je potreba omezovat root?
%Let $\varepsilon > 0$ and $\beta > 2\varepsilon$. Let $(X, Y)$ be an $\varepsilon$-regular pair in a graph $G$ with density $\dens(X, Y) > 3\beta$. Suppose that there are sets $X' \subseteq X, Y' \subseteq Y$, and $X^* \subseteq X'$
%satisfying $|X'| > 4\frac{\varepsilon}{\beta}|X|, |Y'| > 4\frac{\varepsilon}{\beta}|Y|$ and $|X^*| >\beta|X|/2$. Let $(T, r)$ be a rooted tree with colour classes $A$ of order at most $6 \varepsilon|X|$ and $B$ of order at most $6 \varepsilon|Y|$ , $r\in A$. Then there exists an embedding $\tau$ of $T$ in $G$ such that $\tau(r)\in X^*$, $\tau(A)\subseteq X'$ and $\tau(B)\subseteq Y'$.
\end{lemma}

\begin{proof} 
%(sketch) 
%By Lemma~\ref{lem:subdivide} $(X',Y')$ is a $\max(\beta/4,2\varepsilon)$-regular pair of density at least $2\beta$.
We embed vertices of $V(T)\setminus R$ into vertices of $X'$ and $Y'$ which are typical to $Y'$ and $X'$, respectively.
Assume that we have already embedded some part of the tree in this way. We claim that every vertex of this partial embedding in $X$ is incident with more than $\varepsilon |Y|$ vertices typical with respect to $X'$ which have not been used for the partial embedding.
Similarly, every vertex of the partial embedding in $Y$ is incident with more than $\varepsilon |X|$ vertices typical with respect to $Y'$, which have not been used for the partial embedding.

We give arguments only for vertices embedded into $X$, arguments for vertices embedded into $Y$ are symmetric.
For $\varphi(r)\in X$, $r\in R$, the claim follows from the fact that $\varphi(r)$ has  more than $3 \varepsilon |Y|$ neighbors in $Y'$ and out of them, at most $\varepsilon |Y|$ are not typical with respect to $X'$ and at most $\varepsilon |Y|$ have already been used for the partial embedding.
Let $\varphi(v)$, $v\in V(T)\setminus R$ be a vertex of the partially constructed embedding and without loss of generality assume $\varphi(v)\in X'$.
 Since $\varphi(v)$ was chosen to be typical with respect to $Y'$, it is adjacent to at least $(d-\varepsilon)|Y'|$ vertices of $Y'$. Again, out of these vertices, at most $\varepsilon |Y|$ are not typical with respect to $X'$ and at most $\varepsilon |Y|$ have already been used for the partial embedding. Thus, $\varphi(v)$ is typical to at least $(d-\varepsilon)|Y'|-2\varepsilon |Y|>((d-\varepsilon)2\frac{\varepsilon}{\alpha}-2\varepsilon) |Y|$. This is strictly greater than $\varepsilon |Y|$, since $d>3\alpha$ and $\alpha>2\varepsilon$. 
%Analogous argument works for $\varphi(v)\in Y$.

It follows that if $|R|<2$, we can construct embedding greedily.

%, always embedding the next vertex into a vertex which is typical to $X'$ or $Y'$, respectively. This is always possible: if we want to embed a neighbor $w$ of $v$, which is embedded into typical vertex $\varphi(v)\in X'$, since $\varphi(v)$ has degree at least $2\beta|Y'|$ out of which at most $\beta/4|Y'|$ is already used and at most $\beta|Y'|$ is not typical with respect to $X'$ (since $\max(\beta/4,2\varepsilon)\leq\beta$), thus $\phi(v)$ has a neighbor in $Y'$ which is typical and have not been used yet. Symmetrical argument works for embedding a vertex into $X'$.

If $|R|=2$, $R=\{u,v\}$, we first embed vertices of a path connecting $u$ and $v$, starting from $u$ and embedding all but the last two internal vertices of a path into typical vertices, last embedded vertex being $u'$. Then we find an edge between the sets the set $X''$ of vertices of $N(u')$ which are typical to $Y'$ and set $Y''$ of vertices of $N(v)$ which are typical to $X'$.
Since, $X''$ and $Y''$ have size greater than $\varepsilon |X|$ and $\varepsilon |Y|$, respectively by our previous argument, from $\varepsilon$-regularity of $(X,Y)$, it follows that there is an edge $xy$ between $X''$ and $Y''$. We embed the last two internal vertices to $x$ and $y$.

%Since $\varphi(R)$ is a set of at most two vertices typical with respect to $Y'$, there is a set $Y''\subseteq Y'$ of size at least TODO such that every vertex $y\in Y''$ is a neighbor of every vertex $\varphi(R)$ and is typical with respect to $X'$. Thus, we can greedily extend $\varphi$ mapping $F_2$ into $Y''$.

\end{proof}

%\begin{lemma}\label{lem:embed-seeds}

%TODO: lemma na vnorovani seeds (do ultratyp. vrcholu!)

%Lemma 8.6 from https://arxiv.org/pdf/1211.3050.pdf?
%chceme: umime vyplnit par skoro cely, ruzne velke partity
%Let $\beta, \varepsilon > 0$ and $\ell \in \mathbb{N}$ be such that $\beta > 2\varepsilon$. Let $(X, Y)$ be an $\varepsilon$-regular pair with
%$|X| = |Y| = \ell$ 
%of density $d(X, Y) > 3\beta$ in a graph $G$. 

%Let $S$ be a bipartite graph with parts $A$ and $B$ (such that each its com
%Let $(T_1, r_1),(T_2, r_2), \ldots,(T_s, r_s)$ be rooted trees of orders at most $6 \varepsilon\ell$ for all $i \in [s]$. .....
%Let $X^* \subseteq X$ be such that $|X^*| >\sum{i=1}^{s}|V(T_i)| + 50\beta\ell$.
%Then there is are mutually disjoint embeddings $\tau_i$, $i\in [s]$ of $(T_i,r_i)$ in $H$ such that $\tau_i(r_i)\in X^*$ for every $i\in [s]$.

%\end{lemma}

\section{Preliminaries}\label{sec:prelim}

We shall switch freely between a graph~$H$ and its corresponding cluster graph~$\mathbf H$. For example $A\subseteq V(H)$ may as well denote a cluster in an original graph, as $A\in V(\mathbf H)$ a vertex in the corresponding cluster graph. We shall freely use the term \emph{clusters} in a cluster graph~$\mathbf H$ to denote vertices of~$\mathbf H$.  If $\mathcal S\subseteq V(\mathbf H)$ denotes a set of clusters, then $\bigcup \mathcal S$ denotes the corresponding union of vertices in the original graph~$H$. If $A\in V(\mathbf H)$ is a cluster and $\mathcal S\subseteq V(\mathbf H)$ a set of clusters, then $\avdeg(A,\mathcal S)$ denotes the average degree of vertices in $A$ to $\bigcup \mathcal S$ and $\avdeg(A)$ stands short for $\avdeg(A,V(\mathbf H))$. 

We shall use the following notation. The class of all trees of order $k$ is denoted as $\fT_k$. %The oriented weighted graph is a graph $G$ together with a weight function for each of its oriented edges. 
For a graph $G$ and two sets $A \in V(G)$ and $B \in V(G)$ let $G[A,B]$ denote the subgraph of $G$ induced by all edges with one endpoint in $A$ and the other in $B$. 

\begin{definition}%$r$-skew cluster LKS-graph
Let~$r\le 1/2$. We say that a graph $H$ is an {\em $r$-skew LKS-graph with parameters $(k,\eta, \varepsilon, d)$} if there exists a partition  $\{L_1,\ldots, L_{m_L}, S_1,\ldots, S_{m_S}\}$ of $V(H)$ satisfying the following

\begin{enumerate}
    %\item $m_S+m_L<...$\marginpar{?}
    \item $m_L\geq (1+\eta)m_S$, 
    \item all sets $L_i$ have the same size and all sets $S_i$ have the same size,
    \item $r|S_j|=(1-r)|L_i|$ for all $i, j$,
    \item each $(L_i,L_j)$, $i,j\in [m_L]$ and each $(L_i,S_j)$, $i\in [m_L],j\in [m_S]$ is an  $\varepsilon$-regular pair of density either $0$ or at least $d$,
    \item there are no edges inside the sets and no edges between $S_i$ and $S_j$ for $i\neq j$,
    \item average degree of vertices in each $L_i$ is at least $(1+\eta)k$.
    %\item average degree of vertices in each $S_i$ is at most\marginpar{?}
    % \marginpar{muzu chtit rovnou minimalni a maximalni stupen?}
\end{enumerate}

We call the sets $L_i$, $i\in [m_L]$, the $L$-clusters. Similarly, we call the sets $S_i$, $i\in [m_S]$, the $S$-clusters.
%, and use the symbol $L$ to denote the set of all $L$-clusters in $H$.

% Let $\mathbf{H}$ be the weighted graph with vertex set $\{L_1,\ldots, L_{m_L}, S_1,\ldots, S_{m_S}\}$ and with an edge $(L_i,L_j)$, $(L_i,S_j)$ whenever $(L_i,L_j)$ or $(L_i,S_j)$, respectively forms an $\varepsilon$-regular pair of positive density in $H$. 
% The weight of the oriented edge $(C_i, C_j)$ that forms an $\varepsilon$-regular pair of density $d'$ in $H$ is defined as $d' \cdot C_j$ (average number of neighbours in $C_j$ of a vertex in $C_i$). Thus, for any edge $(L_i, L_j)$ we have $w(L_i, L_j) = w(L_j, L_i)$, but for any edge $(L_i, S_j)$ we have $r \cdot w(L_i, S_j) = (1-r) \cdot w(S_j, L_i)$. We call $\mathbf{H}$ the $r$-skew LKS-cluster graph. 
Let $\mathbf{H}$ be the graph with vertex set $\{L_1,\ldots, L_{m_L}, S_1,\ldots, S_{m_S}\}$ and with an edge $(L_i,L_j)$, $(L_i,S_j)$ whenever $(L_i,L_j)$ or $(L_i,S_j)$, respectively forms an $\varepsilon$-regular pair of positive density in $H$. 
Observe that for any edge $(L_i, L_j)$ we have $\avdeg(L_i, L_j) = \avdeg(L_j, L_i)$, but for any edge $(L_i, S_j)$ we have $r \cdot \avdeg(L_i, S_j) = (1-r) \cdot \avdeg(S_j, L_i)$. We call $\mathbf{H}$ the $r$-skew LKS-cluster graph. 
%\marginpar{D: sundala jsem reference na weighted graph, jelikoz se to uz nepouziva}
We use a dot instead of an explicit parameter when the value of the parameter is not relevant in the given context.
\end{definition}

\begin{proposition}\label{prop:degrutratypical}
Let~$H$ be an $r$-skewed LKS graph of order $n$ with parameters $(\cdot, \cdot, \varepsilon, \cdot)$ and let~$\mathbf H$ be its corresponding cluster graph. 
%\marginpar{D: pridala jsem jednu cast. Pozor byla tu chyba}
\begin{enumerate}
    \item Let $C$ and $D$ be an $L$-cluster and  an $S$-cluster of $\mathbf H$, respectively. Then $|C|\le n/|V(\mathbf H)|$ and $|D|\le \frac{n}{r|V(\mathbf H)|}$.
    \item If $v\in V(H)$ is an ultratypical vertex and $\mathcal S\subseteq V(\mathbf H)$, then $\deg(v, \bigcup \mathcal S)\ge \avdeg (C, \mathcal S)-2\sqrt{\varepsilon}n/r$, where~$C$ is the cluster of~$\mathbf H$ containing~$v$.
\end{enumerate}
\end{proposition}

%\marginpar{VT->D: udelali jsme malou zmenu ve vypoctu}
\begin{proof}
${}$
\begin{enumerate}
    \item 
    The first inequality follows from the fact that the size of $L$-clusters is always at most the size of $S$-clusters. Then we compute $|D|=\frac{1-r}{r}|C|\le \frac{n}{r|V(\mathbf H)|}$.

    \item
    If~$v$ is ultratypical, there are at most $\sqrt{\varepsilon}|V(\mathbf H)|$ clusters~$D$ in~$\mathbf H$ such that~$v$ is not typical to~$D$. Denote by $\mathcal D$ the set of those clusters. Then by (1) we have $|\bigcup \mathcal D|\le |\mathcal D|\cdot n/(r|V(\mathbf H)|)\le \sqrt{\varepsilon}n/r$. Then 
    \begin{align*}
        \deg(v, \bigcup \mathcal S)&\ge \deg(v,\bigcup (\mathcal S \setminus \mathcal D))\\
        &\ge \avdeg(C, \bigcup (\mathcal S \setminus \mathcal D))-\varepsilon n\\
        &\ge \avdeg(C, \bigcup \mathcal S )-|\bigcup \mathcal D|-\varepsilon n\\
         &\ge \avdeg(C, \bigcup \mathcal S ) - 2\sqrt{\varepsilon}n/r\;.
    \end{align*}
\end{enumerate}
\end{proof}

 \begin{definition}\cite[Definition 3.3] {HladkyLKS4}
    \label{def:partition}
Let $T\in \mathcal T_{k+1}$ be a tree rooted at~$r$. An
\em $\ell$-fine partition of~$T$ is a
quadruple $(W_A,W_B, \mathcal D_A,
\mathcal D_B)$, where $W_A,W_B\subseteq V(T)$ and $\mathcal D_A$ and $\mathcal D_B$ are families of subtrees of~$T$ such that
\begin{enumerate}
\item  the three sets $W_A$, $W_B$ and $\{V(T^*)\}_{T^*\in\mathcal D_A\cup
\mathcal D_B}$ partition $V(T)$ (in particular, the trees in $T^*\in\mathcal D_A\cup
\mathcal D_B$ are pairwise vertex disjoint),\label{decompose}
\item $r\in W_A\cup W_B$,\label{root}
\item $\max\{|W_A|,|W_B|\}\leq 336k/{\ell}$,\label{few}
\item for $w_1,w_2\in W_A\cup W_B$ the distance $\dist(w_1,w_2)$ is odd if and only if one of them lies in $W_A$ and the other one in $W_B$,\label{parity}
\item $v(T^*)\leq \ell$ for every tree $T^*\in \mathcal D_A\cup
\mathcal D_B$,\label{small}
\item $V(T^*)\cap N(W_B)=\emptyset$ for every $T^*\in
\mathcal D_A$ and $V(T^*)\cap N(W_A)=\emptyset$ for every $T^*\in
\mathcal D_B$,\label{nice}
\item for each tree  $T^*\in \mathcal D_A\cup\mathcal D_B$ $N_T(V(T^*))\setminus V(T^*)\subseteq W_A\cup W_B$,\label{cut:precede}
\item  $|N(V(T^*))\cap (W_A\cup W_B)|\le 2$ for each $T^*\in\mathcal D_A\cup \mathcal D_B$,\label{2seeds}
\item if $N(V(T^*))\cap (W_A\cup W_B)$ contains two distinct vertices $z_1$ and $z_2$ for some  $T^*\in\mathcal D_A\cup
\mathcal D_B$, then $\dist_T(z_1,z_2)\ge 6$,\label{short}
\label{Bsmall}
\end{enumerate}
 \end{definition}

 Here we did not list all properties from~\cite{HladkyLKS4}, only the ones we need.
 \begin{proposition}\cite[Lemma~5.3]{Hladkyn}\label{prop:cutting}
\label{prop:cutfine}
 		Let $T\in \mathcal{T}_{k+1}$ be a tree rooted at a vertex $R$ and let $\ell\in \mathbb{N}, \ell<k$. Then the rooted tree $(T,R)$ has an $\ell$-fine partition.
 \end{proposition}

 Finally, we state two propositions that will be proved in Sections~\ref{sec:comb} and~\ref{sec:main_prop}, respectively. The first proposition says that every LKS-graph contains one the four configurations, while the second proposition asserts that occurrence of these configurations implies containment of a given tree.
 Note that the first proposition is concerned only with the structure of the cluster graph, not the underlying graph, and could be stated in terms of weighted graphs instead.
 
 %we do not use the bold font to denote the cluster graph, because the structure of clusters is not relevant there. 
 
%  \begin{proposition}\label{prop:comb}
%     For each $r'$-skew LKS-cluster graph $H$ with parameters $(k, \eta, \cdot, \cdot)$, let $L$ and $S$, respectively, denote its set of $L$-clusters and $S$-clusters, respectively. 
%   	For any numbers $a_1,a_2,b_1,b_2\in \mathbb N_0$ with $a_2+b_1= \tr k$, $\tr \le r'$, there is a matching~$M$ in~$H[L,S]$ and two adjacent vertices $x,y\in V(H)$ such that, setting 
%   	 $S_M=S\cap V(M)$ and $S_1=\{u\in S\::\: \sum_{v\in V(H)}\vec{w}(uv)\ge (1+\eta) \tr k\}$, one of the four following configurations  occurs.
%   	\begin{itemize}
%   		\item [A)] $\sum_{v\in S_1\cup S_M}\vec{w}(xv)\ge a_2 \cdot(1-\tr) / \tr + \eta k/4$, and $\sum_{v\in L}\vec{w}(yv)\ge \tr k + \eta k/4$,
%   		\item [B)] $\tr a_1 > (1-\tr) a_2$, $\sum_{v\in S_1\cup S_M\cup L}\vec{w}(xv)\ge k + \eta k/4$ and $\sum_{v\in L}\vec{w}(yv)\ge  \tr k + \eta \tr k/4$,
%   		\item [C)] $\tr a_1 \le  (1-\tr) a_2$, $\sum_{v\in S_1\cup S_M\cup L}\vec{w}(xv)\ge k + \eta k/4$ and $\sum_{v\in L}\vec{w}(yv)\ge  b_1+\eta \tr k/4$,
%   		\item [D)] $\tr a_1 \geq (1-\tr) a_2$, $b_1\le \tr^2 k / (1-\tr)$, $\sum_{v\in  S_M\cup L}\vec{w}(xv)\ge k + \eta k/4$ and $\sum_{v\in L}\vec{w}(yv)\ge b_1 + \eta k/4$, and moreover, the neighbourhood of $x$ does not contain both endpoints of any edge from $M$. 
%   	\end{itemize}
%   \end{proposition}
 
 \begin{proposition}\label{prop:comb}
   Let $H$ be a $r'$-skew LKS-graph $\mathbf H$ with parameters $(k, \eta, \cdot, \cdot)$ and let $\mathbf{H}$ be the corresponding cluster graph.
    We denote by $\mathcal L$ and $\mathcal S$, respectively, its set of $L$-clusters and $S$-clusters, respectively. 
   	For any numbers $a_1,a_2,b_1,b_2\in \mathbb N_0$ with $a_2+b_1= \tr k$, $\tr \le r'$, there is a matching~$\mathbf M$ in~$\mathbf H[\mathcal L,\mathcal S]$ and two adjacent clusters $X, Y\in V(\mathbf H)$ such that, setting 
   	 $\mathcal S_M=\mathcal S\cap V(\mathbf M)$ and $\mathcal S_1=\{Z\in \mathcal S\::\: \avdeg(Z)\ge (\tr+r'\eta)  k\}\setminus \mathcal{S_M}$, one of the four following configurations  occurs.
  	\begin{itemize}
  		\item [A)] $\avdeg(X, \mathcal S_1\cup \mathcal S_M)\ge a_2 \cdot(1-\tr) / \tr + \eta k/4$, and $\avdeg(Y, \mathcal L)\ge \tr k + \eta k/4$,
  		\item [B)] $\tr a_1 > (1-\tr) a_2$, $\avdeg(X, \mathcal  S_1\cup \mathcal S_M\cup \mathcal L)\ge k + \eta k/4$ and $\avdeg(Y, \mathcal  L)\ge  \tr k + \eta r' k/4$,
  		\item [C)] $\tr a_1 \le  (1-\tr) a_2$, $\avdeg(X,\mathcal  S_1\cup \mathcal S_M\cup \mathcal L)\ge k + \eta k/4$ and $\avdeg(Y, \mathcal  L)\ge  b_1+\eta r' k/4$,
   		\item [D)] $\tr a_1 \geq (1-\tr) a_2$, $b_1\le \tr^2 k / (1-\tr)$, $\avdeg(X, \mathcal  S_M\cup \mathcal L)\ge k + \eta k/4$ and $\avdeg(Y, \mathcal  L)\ge b_1 + \eta k/4$, and moreover, the neighbourhood of $X$ does not contain both endpoints of any edge from $\mathbf M$. 
  	\end{itemize}
  \end{proposition}

\begin{proposition}
\label{prop:mainembedding2}
For each $\delta, q, d>0$ and $\tr,r'\in \mathbb Q^+$ with $\tr\le r'\le 1/2$ there is $\varepsilon=\varepsilon(\delta, q, d, r')>0$ such that for any $\tilde N_{max} \in \mathbb N$ there is a $ \beta=\beta(\delta, q, r',\varepsilon, \tilde N_{max})>0$ and an $n_0=n_0(\delta, q,\tr,\beta)\in \mathbb{N}$ such that for any $n\ge n_0$ and $k\ge qn$ the following holds.%\marginpar{pozor na stupen neighbours left-out leaves}
Let $\mathcal D=(W_A,W_B,\mathcal D_A,\mathcal D_B)$ be an 	$\beta k$-fine partition of a tree $T\in \mathcal T_k$ with colour classes~$T_1$ and~$T_2$ such that $|T_1|= \tr k$.
%, where $\tr\le r'$.
Let~$H$ be an $r'$-skewed LKS-graph of order~$n$,% $\tr\le r'\le 1/2$, 
with parameters $(k,\delta, \varepsilon, d)$, let~$\mathbf H$ be its corresponding cluster graph with $|V(\mathbf H)|\le \tilde N_{max}$ and $\mathcal{L},\mathcal{S}\subseteq V(\mathbf{H})$ are sets of $L$-clusters and $S$-clusters, respectively.%\marginpar{potrebujeme order of $\mathbf H$?}
Let~$\mathbf M$ be a matching in~$\mathbf H$, let $\mathcal{S}_M=\mathcal{S}\cap V(\mathbf M)$, $\mathcal{S}_1:=\{C\in \fS \setminus V(\mathbf M)\::\ \avdeg(C)\ge (1+\delta) \tr k\}$. Let $A$ and $B$ be two clusters of $\mathbf{H}$ such that $AB\in E(\mathbf{H})$ and  one of the following holds.
\begin{enumerate}[label={\Alph*)}]
	\item\label{it:p11a} $\avdeg(A,\mathcal{S}_1\cup \mathcal{S}_M)\ge a_2 \frac{1-\tr}{\tr} + \delta k$ and $\avdeg(B,\mathcal{L})\ge (\tr+\delta)k$ 
	\item\label{it:p11b} $\tr |V(\mathcal{D}_A)\cap V(T_2)| \ge (1-\tr) |V(\mathcal{D}_A)\cap V(T_1)|$,\\ $\avdeg(A,\mathcal{S}_1\cup \mathcal{S}_M\cup \mathcal{L})\ge (1+\delta)k$, and $\avdeg(B, \mathcal{L})\ge (\tr+\delta)k$,
	\item\label{it:p11c} $\tr |V(\mathcal{D}_A)\cap V(T_2)| \le (1-\tr) |V(\mathcal{D}_A)\cap V(T_1)|$,\\ $\avdeg(A,\mathcal{S}_1\cup \mathcal{S}_M\cup \mathcal{L})\ge (1+\delta)k$, and $\avdeg(B, \mathcal{L})\ge |V(\mathcal{D}_B)\cap V(T_1)|+\delta k$,
	\item\label{it:p11d} $\tr |V(\mathcal{D}_A)\cap V(T_2)| \ge (1-\tr) |V(\mathcal{D}_A)\cap V(T_1)|$, $|V(\mathcal{D}_B)\cap V(T_1)|\le \frac{\tr^2}{(1-\tr)}k$\\
	$\avdeg(A,\mathcal{S}_M\cup \mathcal{L})\ge (1+\delta)k$, $\avdeg(B, \mathcal{L})\ge |V(\mathcal{D}_B)\cap V(T_1)| +\delta k$, and  moreover, the neighbourhood of $A$ does not contain both endpoints of any edge from~$\mathbf M$. 
\end{enumerate}
Then~$T\subseteq H$. 
\end{proposition}

\section{Proof of the theorem}\label{sec:proof}
Suppose $r,q$ and $\eta$ are fixed. If $r=1/2$, then set $r':=r\in \mathbb Q$, $s:=1$, and $t:=2$. Otherwise, let~$\rho:=1/2-r>0$ and $r'\in \mathbb Q$ be such that $r\le r'\le r(1+\frac{\eta\rho  q}{12})$ with $r'=s/t$, $s,t\in \mathbb N$ and  $t\le 12/(\eta \rho q r)$. Observe that $r'\le 1/2$. Let $d:= \frac{\eta^2 q^2r'}{100}$.
 Let $\varepsilon =\min\{\frac{\eta d^2q^2}{40},\frac{1}{t} \varepsilon_{P\ref{prop:mainembedding2}}(\frac{\eta r'q}{400},q,d/2,r')\}$. Lemma \ref{lem:reg} (Szemerédi regularity lemma) with input parameter $\varepsilon_{L\ref{lem:reg}}:= \varepsilon$ and $N_{min}:=1/\varepsilon$ outputs $n_R, N_{max}\in \mathbb N$. Set $\beta:= \beta_{P\ref{prop:mainembedding2}}(\frac{\eta r'q}{400}, q,r', t\cdot \varepsilon, tN_{max})$. 
Let $n_0=\max\{ 2 n_R,2t\cdot N_{max}/\varepsilon, n_{0,P\ref{prop:mainembedding2}}(\frac{\eta r'q}{400}, q,r', \beta)\}$  and let $n\ge n_0$.
Suppose $k\ge qn$ is fixed. 
Let~$G$ be any graph on~$n$ vertices that has at least $rn$ vertices of degree at least $(1+\eta)k$.

We first find a subgraph $H$ of $G$ of size $n''\geq (1-\eta q/2)(1-2\varepsilon)n$
%bylo $n'=n(1-\eta q/2)$ 
which is an $r'$-skew LKS-graph with parameters $(k,\frac{\eta q}{100},  t\cdot \varepsilon, \frac d2)$ and construct the corresponding LKS-cluster graph $\mathbf{H}$.

Erase $\eta \cdot  qn/2$ vertices from the set of vertices that have degree smaller than $(1+\eta)k$ and let~$G'$ be the resulting graph of order $n'=n(1-\eta q/2)$. Observe that for all $v\in V(G')$, we have $\deg_{G'}(v)\ge \deg_{G}(v)-\eta k/2$ and hence at least $rn\ge r'n'(1+\eta q/4)$ vertices of~$G'$ have degree at least~$(1+\eta/2)k$. 
 %Denote this set of vertices by $L$ and its complement by $S$.

We apply Szemer\'edi regularity lemma (Lemma~\ref{lem:reg})  on~$G'$ and obtain an $\varepsilon$-regular equitable  partition $V(G')= V_0\cup V_1\cup\cdots\cup V_N$. Erase all edges within sets~$V_i$, between irregular pairs, and between pairs of  density lower than $d$. Hence, we erase at most 
$N\cdot \binom{n'/N}{2}\le \varepsilon (n')^2/2$ edges within the sets $V_i$, at most $\varepsilon N^2 \cdot \left(\frac{n'}{N}\right)^2=\varepsilon (n')^2$ edges in irregular pairs, and at most $\binom{N}{2}\cdot d\cdot\left(\frac{n'}{N}\right)^2\le \frac{d}{2}\cdot (n')^2$ edges in pairs of density less than~$d$. In total we have thus erased less than $d\cdot (n')^2=\frac{\eta^2 q^2 r'}{100}\cdot (n')^2$ edges. %\marginpar{D: zkorigovala jsem to}
  
Call a set~$V_i$ an $L$-set if the average degree of its vertices is at least $(1+\eta q/4)k$ and otherwise an $S$-set. We have at least $(1+\frac{\eta q}{20})r'N$ $L$-sets. Indeed, during the erasing process, less than $\eta r' q n'/6$ vertices dropped their degree by more than $\eta k/8$. Therefore, now there are at least $(1+\frac{\eta q}{12})r'n'$ vertices of degree at least $(1+3\eta /8)k$. By regularity, in each $S$-set $V_i$ there are at most $\varepsilon |V_i|$ of those vertices, as otherwise they form a subset of~$V_i$ of substantial size and thus the $S$-set $V_i$ would have average degree at least $(1+3\eta /8)k-\varepsilon n'> (1+\eta /4)k$. So we can have at most $\varepsilon n'$ vertices of degree at least 
$(1+3\eta /8)k$ distributed among all $S$-sets and at most $\varepsilon n'$ of them  contained in~$V_0$. Hence, at least $(1+\frac{\eta q}{20})r'n'$  vertices of degree at least $(1+3\eta /8)k$ must be contained in $L$-sets, producing thus at least $(1+\frac{\eta q}{20})r'N$ $L$-sets.

  We subdivide any $L$-set into $t-s$ sets of the same size, which we call \emph{$L$-clusters}, adding at most~$t-s-1$ leftover vertices to the garbage set~$V_0$.
   Similarly, we subdivide any $S$-set into~$s$ sets, which we call \emph{$S$-clusters}. In this way we have $(1-r')|C|=r'|D|$ for any $L$-cluster~$C$, and any $S$-cluster~$D$. By Lemma~\ref{lem:subdivide}, if $(V_i,V_j)$ is $\varepsilon$-regular  and $C\subseteq V_i$ and $D\subseteq V_j$ are~$L$ or~$S$ clusters,  then $(C,D)$, is a $\varepsilon'$-regular pair for $\varepsilon'= t\varepsilon$ with density at least $d':=d-\varepsilon$. 
  Observe that by the choice of~$n_0$, we added in total less than~$t\cdot N\le \varepsilon n'$ vertices to the garbage set~$V_0$.
  We delete at most~$2\varepsilon n'$ vertices of the enlarged set~$V_0$. Any $L$-cluster 
 is a relatively large subset of the $L$-set  it comes from, and thus basically inherits the average degree of the set it comes from. Together with the deletion of the enlarged garbage set, we obtain that each $L$-cluster has now average degree at least $(1+\eta q/4)k-3\varepsilon n'\ge (1+\eta q/5)k$. 
  
   Denote by $m_L$ the number of $L$-clusters and by $m_S$ the number of $S$-clusters. 
   We have $m_L\ge (1+\frac{\eta q}{20})r'N\cdot (t-s)$, as each $L$-set divided in $t-s$ $L$-clusters. Similarly, we obtain $m_S<(1-r') s N$. Therefore,
   \begin{align*}
       m_L&\ge (1+{\eta q}/{100}){m_L}/{2}+(1-{\eta q}/{100})(1+{\eta q}/{20})\cdot r'N\cdot (t-s)/2\\
       &> (1+{\eta q}/{100}){m_L}/{2}+(1+{\eta q}/{100}) \cdot \frac{s}{t}\cdot \frac{m_S}{s(1-s/t)}\cdot (t-s)/2\\
       &= (1+{\eta q}/{100}){m_L}/{2}+(1+{\eta q}/{100})\cdot \frac{m_S}{t-s}\cdot (t-s)/2\\
       &=(1+{\eta q}/{100})(m_L+m_S)/2\;.
    \end{align*}
    %D: of course, we have have something finer than $\frac{\eta q}{100}$, but then it will be divided by $4$, so a whole number looks better
   
   Finally, we delete all edges between $S$-clusters. We denote by~$L$ the set of vertices contained in $L$-clusters and by~$S$ the set of vertices contained in $S$-clusters. 

Let $H$ be the resulting graph. By construction, it is an $r'$-skew LKS-graph of order~$n''$, where $(1-2\varepsilon)n'\le n''\le n'$, with parameters $(k,\frac{\eta q}{100}, \varepsilon', d/2)$. 
 The vertex set of the corresponding cluster graph $\mathbf{H}$ consists of the $L$- and $S$-clusters defined above, with edges corresponding to $\varepsilon'$-regular pairs of density at least~$d/2$ in~$H$. Observe that $|V(\mathbf H)|\le t\cdot N_{max}$.
 
 %We may assume that there is no edge between $S$-clusters. We define now a cluster graph~$\mathbf H$ as follows. The $L$- and $S$-clusters form the vertices of~$\mathbf H$ and two vertices are adjacent in~$\mathbf H$ if there are edges between the corresponding clusters (i.e., they form a regular pair of positive density). We shall call such a cluster graph an \emph{$r'$-skewed LKS-cluster graph} with parameters $(k,\eta q, \varepsilon', d')$.

 After having processed the host graph, we turn our attention to the tree. Let~$T$ be any tree of order~$k$ with colour classes~$T_1$ and~$T_2$ and $|T_1|\le  rk\le r'k$. Pick any vertex $R\in V(T)$ to be the root of~$T$. 
Applying Proposition~\ref{prop:cutting} on~$T$ with parameter $\ell_{P\ref{prop:cutting}}:=\beta k$, we obtain its $\beta k$-fine partition $\mathcal D=(W_A, W_B, \mathcal D_A, \mathcal D_B)$.  Without loss of generality, assume that $W_A\subseteq  V(T_2)$.  Let $\tilde r:= |V(T_1)\setminus W_B|/k$.
 We then apply Proposition~\ref{prop:comb} with $\eta_{P\ref{prop:comb}}:=\eta q/100$, $r'_{P\ref{prop:comb}}:=r'$, $k_{P\ref{prop:comb}}:=k$,   $n_{P\ref{prop:comb}}:=n''$, $H_{P\ref{prop:comb}}:=\mathbf H$, for $vu\in \mathbf{H}$, $a_1:=|V(\mathcal{D}_A)\cap V(T_2)|$, $a_2:= |V(\mathcal{D}_A)\cap V(T_1)|$, $b_1:=|V(\mathcal{D}_B)\cap V(T_1)|$, $b_2:= |V(\mathcal{D}_B)\cap V(T_2)|$, $\tilde r_{P\ref{prop:comb}}:= \tilde r$. We obtain a matching $\mathbf M\subseteq E(\mathbf H)$ and two adjacent clusters $A,B\in V(\mathbf H)$ satisfying one of four configurations.

%$\vec{w}_{P\ref{prop:comb}}(uv)$ is given by the average degree in~$H$ of cluster~$u$ into cluster~$v$

%   For any of these four possible configurations,  Proposition~\ref{prop:mainembedding2} with input $\delta_{P\ref{prop:mainembedding2}}:=\frac{\eta\tilde r q}{400}$, $\tilde N_{max, P\ref{prop:mainembedding2}}:=tN_{max}$,  $H_{P\ref{prop:mainembedding2}}:= H$, $\mathbf H_{P\ref{prop:mainembedding2}}:=\mathbf H$, and further input as in Proposition~\ref{prop:comb},  gives an embedding of~$T$ in $H\subseteq G$, proving Theorem~\ref{thm:result}. 
\marginpar{D: I changed parameters below}
 For any of these four possible configurations,  Proposition~\ref{prop:mainembedding2} with input $\delta_{P\ref{prop:mainembedding2}}:=\frac{\eta r' q}{400}$, $q_{P\ref{prop:mainembedding2}}:= q$, $d_{P\ref{prop:mainembedding2}}:= d/2$, $\varepsilon_{P\ref{prop:mainembedding2}}:=\varepsilon'$, $\tilde N_{max, P\ref{prop:mainembedding2}}:=tN_{max}$,  $H_{P\ref{prop:mainembedding2}}:= H$, $\mathbf H_{P\ref{prop:mainembedding2}}:=\mathbf H$, and further input as in Proposition~\ref{prop:comb},  gives an embedding of~$T$ in $H\subseteq G$, proving Theorem~\ref{thm:result}.

% \subsection{Sketch:}
% \begin{enumerate}
% 	\item
% 	Minimal graph (I THINK NOT NEEDED)
% 	\item
% 	Application of Regularity Lemma (DONE)
% 	\item
% 	Subdivision of clusters (DONE)
% 	\item
% 	Creating the cluster graph (DONE)
% 		\item cutting the tree (DONE)
% 	\item
% 	Combinatorial part (DONE: Vasek, muzes zkontrolovat, ze to je spravne? Vyhovi?)
% 	\item
% 	Embedding part (DONE ?)
% 	\item set up parameters
% \end{enumerate}
%https://arxiv.org/pdf/1408.3870.pdf

\section{Proof of Proposition \ref{prop:comb}}\label{sec:comb}
We will prove Proposition \ref{prop:comb} in several steps. We start by defining the desired matching $\mathbf M$ as well as several other subsets of $\mathbf H$. 

Let $\mathbf M \subseteq \mathbf  H[\mathcal L,\mathcal S]$ be a matching minimising the number of vertices in the set $\mathcal S_0 := \{ X \in \mathcal S : \avdeg(X) < (\tr+r'\eta/2) k \}$.
It follows that $\mathcal S_1 = \mathcal S \setminus (\mathcal S_M \cup\mathcal  S_0)$.

We define $\mathcal B \subseteq V(\mathbf M)$ as the set of those clusters $X$, for which there is an alternating path $P=X_1X_2\dots X_k$, such that $X_1 \in \mathcal S_0$, $X_k=X$, $X_{2i}\in \mathcal L$, $X_{2i+1} \in \mathcal S_M$, $\{ X_{2i}, X_{2i+1} \} \in \mathbf M$. Also let $\mathcal L_B = \mathcal L \cap \mathcal B$ and $\mathcal S_B = \mathcal S_M \cap \mathcal B$. 
Then we define $\mathcal A=V(\mathbf M)\setminus \mathcal B$, $\mathcal L_A = \mathcal L \setminus\mathcal  L_B$, $\mathcal S_A = \mathcal S_M \setminus \mathcal S_B$. 

\begin{figure}
    \centering
    \includegraphics{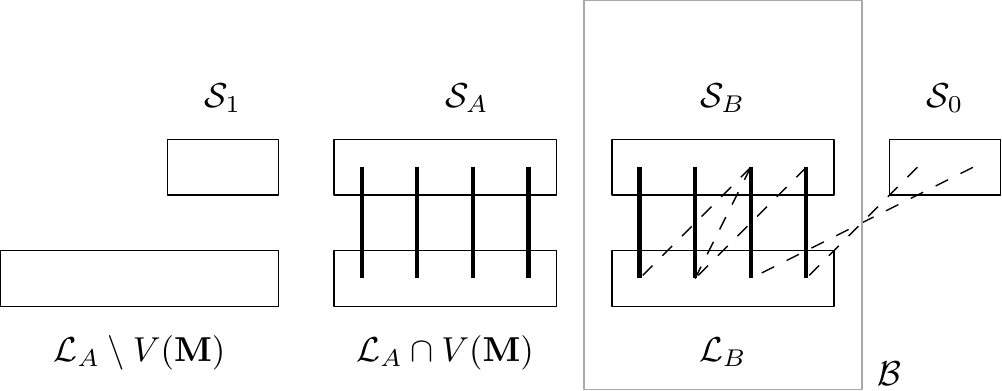}
    \caption{Various subsets of $\mathbf H$ used in the proof of Proposition \ref{prop:comb}. }
    \label{fig:comb}
\end{figure}

\begin{claim}
\label{comb:basic_props}
For all $X\in \mathcal S_B$ we have $\avdeg(X) < (\tr+r'\eta/2) k$. Also, there are no edges between clusters  from $\mathcal L_A$ and $\mathcal S_0 \cup \mathcal S_B$. 
\end{claim}

\begin{proof}
If the first statement was not true, the symmetric difference of $\mathbf M$ and an alternating path between $X$ and a vertex in $\mathcal S_0$ would yield a matching contradicting the choice of~$\mathbf M$ as a matching minimising the size of $\mathcal S_0$.

If the second statement was not true, we would have an alternating path ending at $X$ which is a contradiction with the definition of $\mathcal L_A$. 
\end{proof}

Now we are going to define yet another subsets of $\mathcal L$ based on the average degrees of the clusters. 
$$\mathcal L^* := \{ X \in \mathcal L : \avdeg(X, \mathcal L) \geq (\tr+r'\eta/2) k\},$$
$$\mathcal L^+ := \{X \in \mathcal L \setminus \fL^* : \avdeg(X, \mathcal S_M \cup \mathcal S_1) \ge (1-\tr+\eta /2 )k\}.$$
Next, we define $\mathcal L_A^* := \mathcal L^* \cap \mathcal L_A$ and $\mathcal L_A^+ := \mathcal L^+ \cap \mathcal L_A$. We have $\mathcal L_A^*=\mathcal L_A\setminus \mathcal L_A^+$ by Claim~\ref{comb:basic_props}.%\marginpar{T: proc musi byt + a * disjunktni? To claim neimplikuje. V: to je dobry point. Upravil jsem definici a myslim, ze vse furt funguje}
We define $\mathcal L_B^+$ and $\mathcal L_B^*$ in a similar way. 
Finally, let $$\mathcal N=N(\mathcal L_A^*) \cap \mathcal L.$$

Now suppose that none of the four configurations from statement of the theorem occurs in the cluster graph $\mathbf H$. We are going to gradually constrain the structure of $\mathbf H$ until we find a contradiction.

\begin{claim}
\label{comb:LAprops1}
Let $X$ and $Y$ be two clusters such that $X \in \mathcal L$ and $\avdeg(X, \mathcal S_0) = 0$ and $\avdeg(Y, \mathcal L) \geq (\tr+r'\eta/2) k$. Then $X$ and $Y$ are not connected by an edge. \end{claim}

\begin{proof}
If there is $X \in \mathcal L$ such that $\avdeg(X, \mathcal S_0) = 0$, then we have $\avdeg(X, \mathcal L \cup \mathcal S_1 \cup \mathcal S_M) \ge (1+\eta)k$. 
Now suppose that there is an edge between such a  cluster~$X$ and a cluster~$Y$ with $\avdeg(Y, \mathcal L) \ge (\tr+r'\eta/2) k$. 
If $\biga$, we have found Configuration~B. If, on the other hand, $\bigb$, recall that $b_1 \le a_2 + b_1 = \tr k$, meaning that we have found Configuration~C. 
\end{proof}

\begin{corollary}
\label{comb:LAprops2}
We have:
\begin{enumerate}
\item
$e(\mathcal L_A, \mathcal L^* \cup \mathcal S_1)=0$, thus $\mathcal N$ is a subset of $\mathcal L_B$, 
\item
$\forall X \in \mathcal N : \avdeg(X, \mathcal L) < (\tr+r'\eta/2) k$, 
%\item
%$e(L_A, S_1)=0$, 
\item
$\forall X \in \mathcal S_A : \avdeg(X) = \avdeg(X,\mathcal L) < (\tr+r'\eta/2) k$.
\end{enumerate}
\end{corollary}

\begin{proof}
${}$
\begin{enumerate}
    \item 
    Suppose that there is an edge between $X \in \mathcal L_A$ and $Y \in \mathcal L^* \cup \mathcal S_1$. From Claim \ref{comb:basic_props} we get that $\avdeg(X, \mathcal S_0) = 0$. From the definition of $\mathcal L^*$ and $\mathcal S_1$ we have $\avdeg(Y, \mathcal L) \ge (\tr+r'\eta/2) k$. Thus we can apply Claim \ref{comb:LAprops1} for $X$ and $Y$. 
    \item
    Each vertex $Y \in \mathcal N$ has a neighbour $X \in \mathcal L_A^*$. If $\avdeg(Y,\mathcal L) \ge (\tr+r'\eta/2) k$ we are in the situation of the first part of this claim. 
    \item
    Each vertex $Y \in \mathcal S_A$ is matched to a vertex $X \in \mathcal L_A$. If $\avdeg(Y,\mathcal L) \ge (\tr+r'\eta/2) k$, we are, yet again, in the situation of the first part of the claim. 
\end{enumerate}
\end{proof}

\begin{claim}
\label{comb:Nclusters}
Every cluster  in $\mathcal N$ has average degree at least $(\tr+\eta/2) k$ in $\mathcal S_0$.
\end{claim}

\begin{proof}
Suppose that it is not so. Then we have a cluster $Y \in \mathcal N$ such that 
\begin{align*}
    \avdeg(Y,\mathcal  S_1 \cup \mathcal S_M \cup \mathcal L) 
    &\geq (1+\eta)k - (\tr+\eta/2)  k \\
    &\ge (1-\tr+\eta/2)k\;. 
\end{align*}
Now we consider separately three cases:

\begin{enumerate}
\item
    Suppose that $\tr  a_1 \le (1-\tr ) a_2$. Then either 
    \begin{align*}
        \avdeg(Y, \mathcal L) \geq b_1+\eta k/4,
    \end{align*}
    which leads to the Configuration~C (consider $Y$ and its neighbour in $\mathcal L_A^*$), or we have
    \begin{align*}
    \avdeg(Y, \mathcal S_1 \cup \mathcal S_M) 
    &= \avdeg(Y, \mathcal S_1 \cup \mathcal S_M \cup \mathcal L) - \avdeg(Y, \mathcal L)\\
    &\geq (1-\tr )k+\eta k/2 - (b_1+\eta k/4) \\
    &= \frac{1-\tr}{\tr} \tr k -b_1+\eta k/4 \\
    &= \frac{1-\tr}{\tr}(b_1+a_2)-b_1+\eta k/4 \\
    &= \frac{1-2\tr}{\tr }b_1 + \frac{1-\tr}{\tr} a_2+\eta k/4 \\
    &\geq \frac{1-\tr}{\tr} a_2+\eta k/4,
    \end{align*}
    where we used the bound on the average degree of $Y$ and then the facts that $b_1 + a_2 = \tr k$ and $\tr \le r' \leq 1/2$. This, on the other hand, leads to the Configuration~A (again, consider $Y$ and its neighbour in $\mathcal L_A^*$).

\item
    Suppose that $\biga$ and $b_1 \le \frac{\tr^2}{1-\tr} k$. 
    Following the same considerations as in the previous case we get that either $\avdeg(Y, \mathcal L) \geq b_1 +\eta k/4$ or $\avdeg(Y, \mathcal S_1 \cup \mathcal S_M) \ge \frac{1-\tr}{\tr} a_2+\eta k/4$. The second case leads, again, to the Configuration~A. We now proceed with the first case. 
    
    Let $X$ be a neighbour of $Y$ in $\mathcal L_A^*$. From Claim~\ref{comb:basic_props} we have $\avdeg(X, \mathcal S_0)=0$ and from Corollary~\ref{comb:LAprops2}.1 we have $\avdeg(X,\mathcal S_1)=0$, thus 
    \begin{align*}
        \avdeg(X, \mathcal L \cup \mathcal S_M) 
        = \avdeg(X) 
        \ge (1+\eta)k 
        > k + \eta k/4.
    \end{align*}
    
    Moreover, all the matching edges containing clusters from $\mathcal S \cap N(X)$ must have both ends in the set~$\mathcal A$ because there are no edges between vertices from $\mathcal L \cap \mathcal A$ and $\mathcal S \cap \mathcal B$ (Claim~\ref{comb:basic_props}). On the other hand, all neighbours of $X$ in $\mathcal L$ have to be in $\mathcal B$ (Corollary~\ref{comb:LAprops2}~(1)), so all matching edges containing vertices from $\mathcal L \cap N(X)$ are in $\mathcal B$. Thus, all of the assumptions of Configuration~D for $X$ and $Y$ are satisfied. 
\item
    Finally we are left with the case $\biga$ and $b_1 > \frac{\tr^2}{1-\tr} k$. 
    Note that then we have 
    \begin{align*}
        a_2 
        = \tr k-b_1 
        < \tr k - \frac{\tr}{1-\tr} \tr k 
        = (1 - \frac{\tr}{1-\tr}) \tr k 
        = \left(1-2\tr \right)\frac{\tr}{1-\tr} k. 
    \end{align*}
    
    Now either 
    \begin{align*}
        \avdeg(Y, \mathcal L) 
        &\geq \tr k+r'\eta k/4,
    \end{align*}
    or 
    \begin{align*}
        \avdeg(Y, \mathcal S_1 \cup \mathcal S_M) 
        &= \avdeg(Y,\mathcal  S_1 \cup \mathcal S_M \cup \mathcal L) - \avdeg(Y, \mathcal L)\\
        &\geq (1-\tr )k + \eta k/2 - (\tr k + r'\eta k/4) \\
        &\ge (1-2\tr) k + \eta k/4 \\
        &= \frac{1-\tr}{\tr} \left( 1-2\tr \right) \frac{\tr}{1-\tr} k +\eta k /4 \\
        &\geq \frac{1-\tr}{\tr} a_2+\eta k /4.     
    \end{align*}
    
    The first option leads to Configuration~B while the second one leads to Configuration~A. 
\end{enumerate}
\end{proof}

After restricting the structure of $\mathbf{H}$ we are ready to derive a contradiction by combining several properties of $\mathbf{H}$ together. 
At first we estimate the size of the set $\mathcal L_A$. Recall that we have $|\mathcal L| \ge (1+\eta)|\mathcal S|$, thus $|\mathcal L| > |\mathcal S|$. We also know that $|\mathcal L_B|= |\mathcal S_B|$, because the two sets are matched in $\mathbf M$. This means that 
\begin{equation}
\label{comb:counting_lemma_3}
|\mathcal L_A| = |\mathcal L| - |\mathcal L_B| > |\mathcal S| - |\mathcal S_B| = |\mathcal S_A|+|\mathcal S_0|+|\mathcal S_1|.
\end{equation}
Now we proceed by bounding the size of the set $\mathcal N$. 

\begin{lemma}
\label{comb:counting_lemma}
Suppose that the set $\mathcal L_A^*$ (and thus also $\mathcal N$) is nonempty. Then the following inequality holds:
\begin{equation}
|\mathcal N|(\tr + r' \eta /2)  > |\mathcal S_0|(1-\tr + \eta / 2 ).
\end{equation}
\end{lemma}

\begin{proof}
We estimate the number of edges between $L_A^+$ and $S_A$. For $\mathcal Y,\mathcal Z\subseteq V(\mathbf H)$, we set $\vec{w}(\mathcal Z,\mathcal  Y) :=\sum_{Z\in \mathcal Z}\avdeg(Z, \mathcal Y)$. On one hand we have
\begin{align*}
\vec{w}(\mathcal L_A^+,\mathcal  S_A) =
\sum_{Z\in \mathcal L_A^+}\avdeg(Z, \mathcal S)
\geq
|\mathcal L_A^+|(1-\tr +\eta/2)k, 
\end{align*}
because $\vec{w}(\mathcal L_A^+, \mathcal S_B\cup \mathcal S_0)=0$ (Claim~\ref{comb:basic_props}) and $\vec{w}(\mathcal L_A^+, \mathcal S_1)=0$ (Corollary~\ref{comb:LAprops2}). 
On the other hand we have
\begin{align*}
\vec{w}(\mathcal L_A^+, \mathcal S_A)
&= \sum_{Z\in \mathcal L_A^+,W\in \mathcal S_A }\avdeg(Z, W)\\
&=\sum_{Z\in \mathcal L_A^+,W\in \mathcal S_A } \frac{1-r'}{r'}\avdeg(W, Z)\\
&= \frac{1-r'}{r'} \vec{w}(\mathcal S_A, \mathcal L_A^+)\\
&\le \frac{1-r'}{r'} ( |\mathcal S_A|(\tr+r'\eta/2) k - \vec{w}(\mathcal S_A, \mathcal L_A^*) ) \\
&\le (1-r')  |\mathcal S_A|(1+\eta/2) k - \vec{w}(\mathcal S_A, \mathcal L_A^*)  \\
&\le |\mathcal S_A|(1+\eta/2)(1-\tr )k - \vec{w}(\mathcal L_A^*, \mathcal S_A)\;, 
\end{align*}
because all the clusters from $\mathcal S_A$ (if there are any) have their average degree bounded by $(\tr+r'\eta/2) k$ (Corollary~\ref{comb:LAprops2}), and $\tr \le r'\le 1/2$. After combining the inequalities we get 
\begin{align}
\label{comb:counting_lemma_1}
|\mathcal L_A^+|(1-\tr +\eta/2)k 
\le |\mathcal S_A|(1-\tr +\eta/2)k - \vec{w}(\mathcal L_A^*, \mathcal S_A).
\end{align}

We continue by estimating the number of edges between $\mathcal L_A^*$ and $\mathcal N$. On one hand we have
\begin{align*}
\vec{w}(\mathcal L_A^*,\mathcal  N)  
= \vec{w}(\mathcal N, \mathcal L_A^*)  
\le |\mathcal N|(\tr+r'\eta/2) k
\end{align*}
due to Corollary~\ref{comb:LAprops2}~(2). On the other hand we have 
\begin{align*}
\vec{w}(\mathcal L_A^*,\mathcal  N)  
&= \vec{w}(\mathcal L_A^*, V(\mathbf H)) - \vec{w}(\mathcal L_A^*,\mathcal  S_A) - \vec{w}(\mathcal L_A^*,\mathcal  S_1 \cup \mathcal S_B \cup\mathcal  S_0) \\
&= \vec{w}(\mathcal L_A^*, V(\mathbf H)) - \vec{w}(\mathcal L_A^*, \mathcal S_A) \\
& \ge |\mathcal L_A^*|(1+\eta)k - \vec{w}(\mathcal L_A^*,\mathcal  S_A)\;,
\end{align*}
because there are neither edges between $\mathcal L_A^*$ and $\mathcal S_1$ (Corollary~\ref{comb:LAprops2}~(1)), nor edges between $\mathcal L_A^*$ and $\mathcal S_B \cup \mathcal S_0$ (Claim~\ref{comb:basic_props}) and clusters in $\mathcal L$ have large degree.  

By combining the two inequalities we get
\begin{equation}
\label{comb:counting_lemma_2}
|\mathcal L_A^*|(1+\eta)k - \vec{w}(\mathcal L_A^*, \mathcal S_A)  \le  |\mathcal N|(\tr+r'\eta/2) k.
\end{equation}
%where the last inequality is due to Claim~\ref{comb:LAprops2}~(1) \&~(2). 

%\begin{figure}
%	\centering
%	\includegraphics[scale=.7]{images/doublecounting.pdf}
%	\caption{Clusters in Lemma \ref{comb:counting_lemma}. }
%	\label{comb:doublecounting}
%\end{figure}
Combining Inequalities~\eqref{comb:counting_lemma_1}, \eqref{comb:counting_lemma_2} and~\eqref{comb:counting_lemma_3} in this order we get:%\marginpar{D: tu potrebuji, ze $|\mathcal N|\le |\mathcal L^*_A|$, coz by melo platit V:zmenil jsem tento vypocet + zneni, aby ten predpoklad nebyl potreba}
\begin{align*}
|\mathcal N|(\tr + r' \eta /2)k 
&\ge |\mathcal L_A^*|(1+\eta)k - \vec{w}(\mathcal L_A^*, \mathcal S_A)\\
&\ge  |\mathcal L_A^*|(1+\eta/2)k + |\mathcal L_A^+|(1-\tr +\eta /2)k - |\mathcal S_A|(1-\tr +\eta /2)k\\
&= |\mathcal L_A|(1-\tr +\eta /2)k + |\mathcal L_A^*|\tr k - |\mathcal S_A|(1-\tr +\eta /2)k\\
&= |\mathcal L_A^*|\tr k + (|\mathcal L_A| - |\mathcal S_A|)(1-\tr +\eta /2)k\\
&>  |\mathcal L_A^*|\tr k + (|\mathcal S_0|+|\mathcal S_1|)(1-\tr +\eta /2)k\\
&\ge  |\mathcal S_0|(1-\tr + \eta /2)k 
\end{align*}
which concludes the proof. 
\end{proof}

\begin{corollary}
\label{comb:A_is_empty}
The set $\mathcal L_A$ is empty. 
\end{corollary}

\begin{proof}
Suppose that $\mathcal N$ (and thus also $\mathcal L_A^*$) is nonempty. Then on one hand we have
\begin{equation}
\vec{w}(\mathcal N,\mathcal S_0) 
= \frac{1-r'}{r'} \vec{w}(\mathcal S_0, \mathcal N) 
\leq \frac{1-r'}{r'} |\mathcal S_0|(\tr+r'\eta/2) k
\le  |\mathcal S_0|(1+\eta/2)(1-r')k\;, 
\end{equation}
due to the definition of $\mathcal S_0$ and the fact $\tr \le r'$. On the other hand we have
\begin{align}
\vec{w}(\mathcal N,\mathcal S_0) 
\ge |\mathcal N|(\tr+\eta/2) k
\end{align}
due to Claim~\ref{comb:Nclusters}. After combining the inequalities we get that 
%\marginpar{V: tady delam zmenu}
\begin{equation}
|\mathcal N|(\tr+\eta/2)  \le |\mathcal S_0|(1-r')(1+\eta/2) \le |\mathcal S_0|(1-\tr )(1+\eta/2). 
\end{equation}
Combining with Lemma~\ref{comb:counting_lemma} we get
\[
    |\mathcal S_0|(1-\tr + \eta / 2 ) 
    < |\mathcal N|(\tr + r' \eta /2)
    < |\mathcal N|(\tr + \eta /2)
    \leq |\mathcal S_0|(1-\tr)(1 + \eta/2).
\]
which gives a contradiction, because
\[
    1 - \tr + \eta /2
    > 1 - \tr + \eta /2 - \tr \eta /2
    = (1-\tr)(1+\eta/2).
\]
Thus $\mathcal L_A^*$ and $\mathcal N$ are empty. %\marginpar{V: konec zmeny}

Now suppose that $\mathcal L_A^+=\mathcal L_A$ is nonempty. Then on one hand we have
\begin{align*}
\vec{w}(\mathcal L_A^+, \mathcal S_A)
&= \frac{1-r'}{r'} \vec{w}(\mathcal S_A, \mathcal L_A^+)\\
&< \frac{1-r'}{r'} |\mathcal S_A|(\tr+r'\eta/2) k\\ 
&\le  |\mathcal S_A|(1+\eta/2)(1-r')k\;,
\end{align*}
because of Corollary~\ref{comb:LAprops2}~(3) and on the other hand we have 
\begin{align*}
\vec{w}(\mathcal L_A^+, \mathcal S_A) &= \vec{w}(\mathcal L_A^+, \mathcal S_M\cup \mathcal S_1)\\
&\geq |\mathcal L_A^+|(1-\tr+\eta/2)k\\
&= |\mathcal L_A|(1-\tr+\eta/2)k \\
&\ge  |\mathcal S_A|(1-r'+\eta/2)k\\
&>  |\mathcal S_A|(1+\eta/2)(1-r')k\;,
\end{align*}
where we used the definition of $\mathcal L^+_A$, Corollary~\ref{comb:LAprops2}~(1), Claim~\ref{comb:basic_props}, Inequality~\eqref{comb:counting_lemma_3}, and the fact that $\tr \le r'$. 

Combining the inequalities gives a contradiction. 
Thus, the set $\mathcal L_A$ has to be empty. 
\end{proof}

From Corollary~\ref{comb:A_is_empty} it follows that all $L$-clusters are in~$\mathcal L_B$ and  thus are matched to $\mathcal S_M$, i.e., $|\mathcal L|=|\mathcal S_M|\le |\mathcal S|$, which contradicts our assumption that $|\mathcal L| > |\mathcal S|$.

\section{Embedding}\label{sec:embedding}

%\renewcommand{\aone}{a_1}
%\renewcommand{\atwo}{a_2}
%\renewcommand{\bone}{b_1}
%\renewcommand{\btwo}{b_2}

%\begin{defn}[$L$-fine partition]
%Let $T/in trees(k)$ be a tree rooted at $r$. An $L$-fine partition of $T$ is a quadruple $(WA, WB, SA, SB)$, where $WA$, $WB \subseteq V(T)$ and $SA$ and $SB$ are families of subtrees
%of $T$ such that
%
%\begin{enumerate}
%	\item
%	the three sets $WA$, $WB$ and $\{V (T)\}_{T^* \in S_A \cup S_A}$ partition $V (T)$ (in particular, the trees in $T^* \in S_A \cup S_A$ are pairwise vertex disjoint),
%	\item
%	$r \in W_A \cup W_A$,
%	\item 
%	$\max{|WA|, |WB|} \leq \frac1\beta$,
%	\item
%	for $w_1, w_2 \in W_A \cup W_B$ the distance $\dist(w_1, w_2)$ is odd if and onlz if one of them lies in $W_A$ and the other one in $W_B$, 
%	\item
%	$v(T^*) \leq \beta k$ for every tree $T^* \in S_A \cup S_B$,	
%	\item
%	 $v(T^*) \cap N(W_B) = \emptyset$ for every tree $T^* \in S_A$ and $v(T^*) \cap N(W_A) = \emptyset$ for every tree $T^* \in S_B$, 
%	 \item
%	 $|N(V(T^*)) \cap (W_A \cap W_B)| \leq 2$ for each $T^* \in S_A \cup S_B$.  
%
%\end{enumerate}
%\end{defn}

We call a pair $(F,R)$ an \emph{anchored $\tau$ forest} if $F$ is a forest (possibly consisting of a single tree), $R\subseteq V(F_1)$, where $F_1$ is one of the colour classes of $F$, $F-R$ decomposes into components of size at least two and at most~$\tau$, each component $K$ in $F-R$ is adjacent in~$F$ to at least one and at most two vertices from $R$ and each two vertices in $R$ are of distance at least~$4$. We shall use the notation $K\in F-R$ to denote that the tree~$K$ is one of the components of $F-R$.

First we state a proposition that will allow us to use matching edges in our $r$-skewed LKS-cluster graph to embed part of our tree $T$.

Specifically, in Proposition \ref{prop:embed-SM} we are given an anchored forest~$(F,R)$, an $r$-skewed LKS-graph which contains a cluster~$A$ with some nice average degree to some $L-S$-matching, and an injective mapping of~$R$ on ultratypical vertices of~$A$. We want to extend it to an embedding of~$F$.  
  
\begin{proposition}\label{prop:embed-SM}
For all $\eta,d>0$ and $r\in \mathbb Q^+$, $0<r\le 1/2$, there is an $\varepsilon=\varepsilon(\eta, d,r)>0$ such that for any $\tilde N_{max} \in \mathbb N$ there is a $\beta=\beta(\eta, r, \varepsilon, \tilde N_{max})>0$ such that for all $n\in \mathbb N$ the following holds.%\marginpar{D: snazim se se zbavit $k$. vypada, ze $n$ nemusi byt veliky?}
Let $(F, R)$ be an anchored $\beta n$-forest%\marginpar{pokud $(F,R)$ je $\beta k$-anchored forest, pak je take $\beta n$-anchored forest}
with colour classes $F_1$ and $F_2$ such that $R\subseteq F_2$ and for each component $K\in F-R$, we have $|F_1\cap K|\le |F_2\cap K|$. 
Let $H$ be an $r$-skewed LKS-graph of order~$n$ with parameters $(\cdot,\cdot,\varepsilon,d)$ with a corresponding cluster graph~$\mathbf H$ of order at most~$\tilde N_{max}$.
Let $U\subseteq V(H)$ and let $\mathbf M\subseteq E(\mathbf H)$ be a matching in $\mathbf H$ between $L$-clusters and $S$-clusters.

If for $A\in V(\mathbf H)$ we have \[\avdeg(A,\mathcal{S}\cap V(\mathbf M) )\ge \frac{1-r}{r}|F_2|+\sum_{C\subseteq  S\::\: CD\in \mathbf M}\max\{|U\cap C|, \frac{1-r}{r}|U\cap D|\} +\eta n\;,\]
%\marginpar{D: apply with $\eta_\psm=q\cdot \delta/2$} 
then for any injective mapping of~$R$ on ultratypical vertices of~$A$, there is an embedding~$\varphi$ of~$F$ avoiding~$U$ and extending this mapping such that $\varphi(V(F_1))\subseteq S\cap V(\mathbf M)$, $\varphi(V(F_2)\setminus R)\subseteq L\cap \bigcup V(\mathbf M)$, and $V(F_2)$ are mapped on ultratypical vertices\marginpar{zkorigovat dukaz}. Moreover, for any cluster $C\in V(\mathbf H)$ where we embedded vertices from $F-R$ it holds that $|C\setminus (U\cup \varphi(F))|\ge r \eta/8|C|$.
\end{proposition}

Next, we state a proposition allowing us to use high average degree of some clusters to embed further part of our tree~$T$. 

Specifically, in Proposition \ref{prop:embed-LS} we are given an anchored forest~$(F,R)$, an $r$-skewed LKS-graph which contains a cluster~$A$ with big enough average degree to a set of clusters with high average degree, and an injective mapping of~$R$ on ultratypical vertices of~$A$. We want to extend it to an embedding of~$F$.

When using the proposition, we always set $\mathcal B$ to be the set of $L$-clusters in (1) and the set of $S_1$-clusters in (2).

% \begin{proposition}\label{prop:embed-LS}SOMETHING ABOUT PARAMETERS: $\tau$, $\eta$, $\varepsilon$, $d$, $k\ge qn$\\
% Let $(F, R)$ be an anchored $\tau$-forest with colour classes $F_1$ and $F_2$ such that $R\subseteq F_2$. Let $\mathbf H$ be an $r$-skewed LKS-cluster graph with parameters $(k,10\eta,\varepsilon,d)$ of an original graph~$G$ of order~$n$, let $U\subseteq V(G)$, and  $A\in V(H)$. If $\mathcal B\subseteq V(\mathbf H)$ is a set of clusters $B$ for which
% \begin{enumerate}
% \item  $\avdeg(B)\ge |F_1|+|F_2|+|U|+\eta k$ and we have $\avdeg(A, \bigcup \mathcal B)\ge |F_1|+|\bigcup \mathcal B\cap U|+\eta k$, or
% \item $\avdeg(B, V(G)\setminus \bigcup \mathcal B)\ge |F_2|+|U|+\eta k$ and $\avdeg(A, \bigcup \mathcal B)\ge |F_1|+|\bigcup \mathcal B\cap U|+\eta k$,
% \end{enumerate}
% then for any injective mapping of~$R$ on ultratypical vertices of~$A$, there is an embedding~$\varphi$ of~$F$ in~$G$ avoiding~$U$ and extending this mapping such that 
%  $\varphi(V(F_1))\subseteq \bigcup \mathcal B$.
% \end{proposition}

\begin{proposition}\label{prop:embed-LS}
For all $\eta,d >0$ and $0<r\le 1/2$, there is an $\varepsilon=\varepsilon(\eta,d,r)>0$  such that for any $\tilde N_{max} \in \mathbb N$ there is a $\beta=\beta(\eta, r, \varepsilon,\tilde N_{max})>0$ such that for all $n\in \mathbb N$ the following holds. 
Let $(F, R)$ be an anchored $\beta n$-forest with colour classes $F_1$ and $F_2$ such that $R\subseteq F_2$.
Let $H$ be an $r$-skewed LKS-cluster graph with parameters $(\cdot ,\cdot,\varepsilon,d)$ of order~$n$ with an associated cluster graph $\mathbf H$ of order at most~$\tilde N_{max}$. Let $U\subseteq V(H)$ and let $\mathcal B\subseteq V(\mathbf H)$ be a set of clusters. Let $\varphi: R\rightarrow A$ with $A\in V(\mathbf H)$ be an injective mapping on ultratypical vertices.

\begin{enumerate}
\item 
If 
$\avdeg(A, \mathcal B)\ge |F_1|+ |\bigcup \mathcal B\cap U|+\eta n$,
then we can extend $\varphi$ to $N(R)$ so that $\varphi(N(R))$ are ultratypical vertices in $\bigcup \mathcal B\setminus U$ and find a set $W=W_1\dot\cup W_2\dot\cup \ldots\subseteq \bigcup \mathcal B\setminus (U\cup \varphi(R\cup N(R)))$ of reserved vertices such that $|W_i|=|(F_1\cap K_i)\setminus N(R)|$, with $K_i\in F-R$ and such that~$W_i$ lies in the same cluster as $\varphi(K_i\cap N(R))$ and for each cluster $C\in \mathcal B$ with $C\cap \varphi(N(R))\neq \emptyset$ we have $|C\setminus (U\cup W\cup \varphi(N(R)))|\ge r\eta /8\cdot |C|$ .

Moreover, for any set $\tilde U\subseteq V(G)\setminus (U\cup W\cup \varphi(R\cup N(R)))$, for which  $\avdeg(B)\ge |F_1|+|F_2|+|U\cup \tilde U|+\eta n$ for each $B\in \mathcal B$ and such that for any $C\in V(\mathbf H)$ with $C\cap \tilde U\neq\emptyset$ we have $|C\setminus (U\cup W\cup \tilde U\cup \varphi(N(R)))|\ge r\eta /8\cdot |C|$,  we can further extend~$\varphi$ to the whole~$F$ avoiding $U\cup \tilde U$ such that $\varphi(F_1)\subseteq \bigcup \mathcal B$. Moreover, the extension~$\varphi$ is such that for any cluster $C\in V(\mathbf H)$ with $C\cap \varphi(F-(R\cup N(R))\neq\emptyset$, we have  $|C\setminus (\tilde U\cup U)|\ge  r\eta/8\cdot |C|$.
\item If  $\avdeg(A, \mathcal B)\ge |F_1|+|\bigcup \mathcal B\cap U|+\eta n$ and 
$\avdeg(B, V(\mathbf H)\setminus \mathcal B)\ge |F_2|+|U|+\eta n$ for each $B \in \mathcal B$, then we can extend~$\varphi$ to~$F$ in~$V(G)$  avoiding~$U$ and such that 
 $\varphi(V(F_1))\subseteq \bigcup \mathcal B$, $\varphi(V(F_2))\subseteq \bigcup N_{\mathbf H}(\mathcal B)\setminus \mathcal B$, and $V(F_2)$ are mapped on ultratypical vertices. Moreover, the embedding~$\varphi$ is such that for any cluster $C\in V(\mathbf H)$ with $C\cap \varphi(F-R)\neq\emptyset$, we have  $|C\setminus (\varphi(F)\cup U)|\ge r \eta/8\cdot |C|$.
\end{enumerate}
\end{proposition}

We at first prove Proposition \ref{prop:embed-SM}.

\begin{proof}[Proof of Proposition~\ref{prop:embed-SM}]
Given $\eta,d>0$ and $r\in \mathbb Q$ set  $\varepsilon=\min\{(\frac{\eta r}{12})^2, \frac{d r\eta}{100}\}$.%\marginpar{D: correction because of Prop 7 V: precetl jsem a dava mi to smysl} 
For any  $\tilde N_{max}\in \mathbb N$ set $\beta=\frac{\varepsilon r\eta}{4\tilde N_{max}}$.

We shall define a set $\tilde U$ of vertices used for the embedding process. At the beginning $\tilde U=\varphi(R)$. 
% Analogously as for the set $U$, we define the set $\ghost (\tilde U)$. Observe that $|\ghost (\tilde U)|\le  \frac{1-r}{r}|R|$.\marginpar{D: We should be careful with floors and ceilings}  
At any time of the embedding process, let $\varphi$ be the partial embedding of $F$. We shall embed one by one each component $K\in F-R$. The embedding~$\varphi$ will be defined in such a way that $\varphi(K\cap F_1)\subseteq S$ and $\varphi(K\cap F_2\setminus R)\subseteq L$.
During the whole embedding process, we shall ensure that the following holds
% \begin{equation}
%     \avdeg (A, S\cap V(M))\ge \frac{1-r}{r}(|F_2| - |\varphi(F_2)|)+|\ghost (\tilde U)|+|\ghost(U)|+\eta k\;.
% \end{equation}
\begin{equation}
    \avdeg (A, \mathcal S\cap V(\mathbf M))\ge \frac{1-r}{r}(|F_2| - |\varphi(F_2)|)+\sum_{C\subseteq S\::\: CD\in \mathbf M}\max\{|(U\cup \tilde U)\cap C|,\frac{1-r}{r}|(U\cup \tilde U)\cap D|\}+\eta n\;.
\end{equation}
This holds at the beginning when $\tilde U=R$. 

 For each next $K\in F-R$ to be embedded, let $R_K$ be the vertices in $R$ adjacent to $K$ (at least one, at most two). Let $\mathcal S'\subseteq \mathcal S \cap V(\mathbf M)$ be such that both~$\varphi(R_K)$ are typical to each cluster $C\in \mathcal S'$. By Lemma~\ref{lem:ultratypical} we have that $|\mathcal S \cap \mathbf M|-|\mathcal S'| \leq 2\sqrt{\varepsilon}|V(\mathbf H)|$ and thus similarly as in the proof of Proposition~\ref{prop:degrutratypical} we can calculate for $x_i\in R_K$, $i=1,2$ that $\deg(\varphi(x_i), \bigcup \mathcal S')\ge \avdeg(A, \mathcal S\cap V(\mathbf M))-3\sqrt{\varepsilon}n/r$ %\marginpar{D: correction because of Prop 7}%\marginpar{V: zmena konstanty 5 na 3}  
 and thus 
% \begin{align*}
%     \deg(\varphi(R_K), \bigcup \mathcal S')&\ge \avdeg(A, \bigcup \mathcal S')-\varepsilon |\bigcup \mathcal S'|\\
%     &\ge \avdeg(A, \bigcup \mathcal S_M)- |\bigcup (\mathcal S_M\setminus \mathcal S')|- \varepsilon n\\
%     &\ge \avdeg(A, S\cap V(M))- |\mathcal S_M\setminus \mathcal S'|\cdot 2n/|V(\mathbf H)| - \varepsilon n\\
%   % &\ge \frac{1-r}{r}|F_2|+|\ghost(U)|+\eta k - 2\sqrt{\varepsilon}\cdot 2n-\varepsilon n\\
%     &\ge  \frac{1-r}{r}(|F_2| - |\varphi(F_2)|)+|\ghost(\tilde U)|+|\ghost(U)|+\eta k - 2\sqrt{\varepsilon}\cdot 2n-\varepsilon n\\
%     &\ge \frac{1-r}{r}(|F_2| - |\varphi(F_2)|)+|\ghost(\tilde U)|+|\ghost(U)|+3\eta k/4\;.
% \end{align*}
\begin{align*}
    \deg(\varphi(x_i), \bigcup \mathcal S')%&\ge \avdeg(A, \bigcup \mathcal S')-\varepsilon |\bigcup \mathcal S'|\\
%     &\ge \avdeg(A, \bigcup \mathcal S_M)- |\bigcup (\mathcal S_M\setminus \mathcal S')|- \varepsilon n\\
%     &\ge \avdeg(A, S\cap V(M))- |\mathcal S_M\setminus \mathcal S'|\cdot 2n/|V(\mathbf H)| - \varepsilon n\\
%   % &\ge \frac{1-r}{r}|F_2|+|\ghost(U)|+\eta k - 2\sqrt{\varepsilon}\cdot 2n-\varepsilon n\\
    &\ge  \frac{1-r}{r}(|F_2| - |\varphi(F_2)|)+\sum_{C\subseteq \mathcal S\::\:CD\in \mathbf M}\frac{1-r}{r}|\tilde U\cap D|\\
    &+\sum_{C\subseteq \mathcal S\::\:CD\in \mathbf M}\max\{|U\cap C|,\frac{1-r}{r}|U \cap D|\}+\eta n - 3\sqrt{\varepsilon}n/r\\
    &\ge \frac{1-r}{r}(|F_2| - |\varphi(F_2)|)+ \sum_{C\subseteq \mathcal S\::\: CD\in \mathbf M}\max\{|(U\cup \tilde U)\cap C|,\frac{1-r}{r}|(U\cup \tilde U)\cap D|\}+3\eta n/4\;\\
    &\ge \sum_{C\subseteq \mathcal S'\::\:  CD\in \mathbf M} \left( \max\{|(U\cup \tilde U)\cap C|,\frac{1-r}{r}|(U\cup \tilde U)\cap D|\}+3\eta n/(4|\mathcal S'|) \right)\;.
\end{align*}
%\marginpar{potrebujeme $\varepsilon<(\frac{\eta}{20})^2$}
Then there is a $C\in \mathcal S'$ with $CD\in \mathbf M$ such that 
\begin{align*}
    \deg(\phi(x_i), C) \ge \max\{|(U\cup \tilde U)\cap C|,\frac{1-r}{r}|(U\cup \tilde U)\cap D|\}+3\eta n/(4|\mathcal S'|)\;.
\end{align*}
Thus, 
\begin{align}
\nonumber|C|-\max\{|(\tilde U\cup U)\cap C|, \frac{1-r}{r}|(\tilde U\cup U)\cap D|\}
&\ge \deg(\varphi(x_i, C)) -\max\{|(\tilde U\cup U)\cap C|, \frac{1-r}{r}|(\tilde U\cup U)\cap D|\}\\
\nonumber&\ge 3\eta n/(4|\mathcal S'|)\\
\nonumber& \ge \frac{1-r}{r}\beta n+\eta n/(2|V(\mathbf H)|)\\
&\ge |F_1\cap K|+\eta r |C| / 2\;,\label{al:zbytekC}
\end{align}
where the third inequality follows from the definition of $\beta$ and the last inequality follows from Proposition \ref{prop:degrutratypical} (1). Similarly we have
%\marginpar{poterebujem $\beta<\frac{r\eta}{4\tilde N_{max}}$}
\begin{align}\nonumber|D\setminus (\tilde U\cup U)|&\ge \frac{r}{1-r}(|C|-\max\{|(\tilde U\cup U)\cap C|, \frac{1-r}{r}|(\tilde U\cup U)\cap D|\})\\
\nonumber &\ge \beta n+\frac{r}{1-r}\eta n/(2|V(\mathbf H)|)\\ 
&\ge |F_2\cap K|+ \eta r |D| / 2\;,\label{al:zbytekD}
\end{align}
where we again use the definition of $\beta$ and Proposition~\ref{prop:degrutratypical} (1).

In particular, in the neighbourhood of each vertex $u_i\in \varphi(R_K)$, $i=1,2$, there are at least $|F_1\cap K|$ unused  vertices of $C\setminus U$ that are typical w.r.t. $D\setminus (\tilde U\cup U)$. Let $\varphi(N(x_i)\cap K)=v_i$, $i=1,2$, be  such vertices. 
Hence,  
\[
\deg(v_i, D\setminus (\tilde U\cup U))\ge (d-\varepsilon)|D\setminus (\tilde U\cup U)|\ge (d-\epsilon)r\eta|D| / 8 > 3\varepsilon |D|
\]
for $i=1,2$.
%\marginpar{potrebuji $d>\frac{24\varepsilon}{r\eta}$}
Observe that $|K|\le \beta n< \frac{\varepsilon rn}{|V(\mathbf H)|}\le \varepsilon \min\{|C|,|D|\}$.
%\marginpar{D: pouzivame: $\beta<\varepsilon r/M$, $d>3\alpha$}
We can thus use Lemma~\ref{lem:tree-emb} with $T_{L\ref{lem:tree-emb}}:= K$, $X'_{L\ref{lem:tree-emb}}:=  C\setminus (\tilde U\cup U)$, $Y'_{L\ref{lem:tree-emb}}:=  D\setminus (\tilde U\cup U)$, $R_{L\ref{lem:tree-emb}}:= \{N(x_i),\; i=1,2\}$ $\varepsilon_{L\ref{lem:tree-emb}}:= \varepsilon$, $\alpha_{L\ref{lem:tree-emb}}:=\frac{16\varepsilon}{\eta r}$, and $d_{L\ref{lem:tree-emb}}:= d$ to embed~$K$ in $C\cup D$
with $\varphi(F_1\cap K)\subseteq C\setminus (\tilde U\cup U)\subseteq \mathcal S$ and $\varphi(F_2\cap K\setminus R)\subseteq D\setminus (\tilde U\cup U)\subseteq \mathcal L$.
Add~$\varphi(K)$ to~$\tilde U$. 
%\marginpar{?nepamatuji si, proc jsem toto psala?}(As $\beta k<r\eta/8|C|$ and $\beta k<r\eta /8|D|$) 
From~\eqref{al:zbytekC} and~\eqref{al:zbytekD}, we now have that $|C\setminus (\tilde U\cup U)|\ge  r\eta/8|C|$, and  $|D\setminus (\tilde U\cup U)|\ge  r\eta/8|D|$. Observe also that for the partial embedding~$\varphi$ we have
\begin{align*}
\avdeg(A,\mathcal S\cap V(\mathbf M) )&\ge \frac{1-r}{r}|F_2|+\sum_{C\subseteq  \mathcal S\::\: CD\in \mathbf M}\max\{|U\cap C|, \frac{1-r}{r}|U\cap D|\} +\eta n\\
    &\ge \frac{1-r}{r}((|F_2| - |\varphi(F_2)|)+|\tilde U\cap L|)+\sum_{C\subseteq \mathcal S\::\: CD\in \mathbf M}\max\{|U\cap C|, \frac{1-r}{r}|U\cap D|\}+\eta n\\
    &\ge  \frac{1-r}{r}(|F_2| - |\varphi(F_2)|)+\sum_{C\subseteq \mathcal S\::\: CD\in \mathbf M}\max\{|(U\cup\tilde U)\cap C|, \frac{1-r}{r}|(U\cup\tilde U)\cap D|\}+\eta n\;,
\end{align*}
where the last inequality comes from the fact that $|F_1\cap K|\le |F_2\cap K|$ for all $K\in F-R$, and that the embedding $\varphi$ was defined in such a way that $\varphi(F_1)\subseteq \mathcal S$ and $\varphi(F_2\setminus R)\subseteq \mathcal L$.

Proceeding in the same way for every $K\in F-R$, we extend $\varphi(R)$ to the whole anchored forest~$F$ in such a way that $\varphi(F_1)\subseteq \mathcal S\cap V(\mathbf M)$, $\varphi(F_2\setminus R)\subseteq \mathcal L\cap V(\mathbf M)$, and for each cluster $C\in V(\mathbf H)$ with $C\cap \varphi(F-R)\neq \emptyset$ we have  $|C\setminus (\tilde U\cup U)|\ge  r\eta/8|C|$.

\end{proof}

We conclude this section by proving Proposition \ref{prop:embed-LS}. 

\begin{proof}[Proof of Proposition~\ref{prop:embed-LS}]Given~$\eta,d>0$ and $r\in \mathbb Q^+$, let $\varepsilon:=\min\{\left(\frac{\eta r}{12}\right)^2,\frac{dr\eta}{100}\}$. Then for any $\tilde N_{max}\in \mathbb N$, set $\beta =\frac{r\eta\varepsilon}{4\tilde N_{max}}$. 

We shall prove only the more difficult Case~(1). Case~(2) can be proven either analogously, or can be much simplified as $F_2$ will be mapped outside of $\mathcal B$ and thus does not need any reservation or cause any difficulties in embedding $F_1$.

We define a set~$W=W_1\cup W_2 \cup \ldots$ of reserved vertices by setting $W=\emptyset$ at the beginning and progressively adding vertices to it. Also we shall define the set $\tilde W$ as the set of vertices used by the partial embedding of $F-R$. Hence at the beginning we have $\tilde W =\emptyset$.  
Suppose that for some $s$, we have already embedded~$K_j\in F-R$, for $j\leq s$. Suppose that $W=W_1\cup \cdots \cup W_s$ is the corresponding set of reserved vertices, i.e., $|W\cup \tilde W|= \sum_{j=1}^s |K_j| $.
For the next component $K_{s+1}\in F-R$ to be embedded, let $R_{s+1}$ be the set of vertices in $R$ adjacent to $K_{s+1}$ (at least one, at most two). 
Let $\mathcal B'\subseteq \mathcal B$ be such that $\varphi(R_{s+1})$ are typical to each cluster $C\in \mathcal B'$. By Lemma~\ref{lem:ultratypical} we have that $|\mathcal B\setminus \mathcal B'| \le 2\sqrt{\varepsilon}|V(\mathbf H)|$ and thus similarly as in Proposition \ref{prop:degrutratypical} we get for $x\in R_{s+1}$ that
\begin{align*}
\deg (\varphi(x), \bigcup \mathcal B')
%&\ge \avdeg(A, \bigcup \mathcal B')-\varepsilon |\bigcup \mathcal B'|\\
&\ge \avdeg(A, \mathcal B) - 3\sqrt{\varepsilon}n/r\\
&\ge \sum_{j=1}^{s+1}|K_j\cap F_1|+|\bigcup \mathcal B \cap U|+\eta n - 3\sqrt{\varepsilon}n/r\\
&\ge |\bigcup_{j=1}^sW_j|+|\tilde W|+|K_{s+1}\cap F_1|+|\bigcup \mathcal B \cap U|+3\eta n/4\\
&\ge |K_{s+1}\cap F_1|+|W|+|\tilde W|+|\bigcup \mathcal B \cap U|+3\eta n/4,  \end{align*}

Hence there is a cluster $B\in \mathcal B'$ (not depending on the choice of vertex~$x$ in~$R_{s+1}$) such that 
\[
    \deg(\varphi(x),B\setminus (U\cup W\cup \tilde W))
    \ge 3\eta n/(4|\mathcal B'|) 
    \ge \frac{\eta n}{4\tilde N_{max}}+\frac {\eta n}{2|V(\mathbf H)|}
    > \beta n+\eta r /2\cdot|B|.
\]

In particular, in the neighbourhood of each vertex of $R_{s+1}$ there are at least $|K_{s+1}|$ unused and unreserved ultratypical vertices in $B\setminus U$. For each $x\in R_{s+1}$, map its neighbor in $K_{s+1}$ to one of  these vertices and add the image to $\tilde W$. Choose a set of vertices of size $|K_{s+1}\cap F_1\setminus N(R)|$ in $B\setminus (U\cup W\cup \tilde W)$  and add it to $W_{s+1}$ (i.e., also to $W$). Observe that $|B\setminus (U\cup W\cup \tilde W)|\geq \eta r /8\cdot|B|$. We proceed in the same way for every $K\in F-R$. 

When we have embedded $N(R)\cap K$ of the last component $K\in F-R$, we have obtained an embedding of $N(R)$ and a reservation set $W=W_1\cup  W_2\cup \ldots$ for $F_1\setminus N(R)$ such that~$W_j$ lies in the same cluster as $\varphi(N(R)\cap K_j)$ does, and in such a way, that for any cluster $B$ where we embedded vertices from~$N(R)$ (and possibly reserved space), we still have at least some unused and unreserved vertices, i.e., 
$|B\setminus (U\cup W\cup \tilde W)|\ge \eta r /8\cdot|B|$. 

Now we shall proceed with the 'moreover' part, i.e., the embedding the left-over of the trees $K_j\in F-R$.  Let $u,v$ in cluster~$B$ be the images of $K_j\cap N(R)$ (alternatively there is only one such image). Set $W_j=\emptyset$ (and thus remove from~$W$ a set of vertices of the size $|K_j\cap F_1\setminus N(R)|$). Similarly as above, we find $D\in V(\mathbf H)$ such that 

\[\deg(u, D\setminus (U\cup W\cup \tilde W\cup \tilde U)\ge |K_j\cap F_2|+\eta r /8\cdot |D|,\]
and similarly
\[\deg(v, D\setminus (U\cup W\cup \tilde W\cup\tilde U)\ge |K_j\cap F_2|+\eta r /8\cdot |D|.\]

 As we have  $|B\setminus (U\cup W\cup  \tilde W\cup \tilde U)|\ge |K_j\cap F_1\setminus N(R)|+r\eta /8\cdot |B|$ and  $r\eta/8>3\varepsilon$, we may use Lemma~\ref{lem:tree-emb} with $T_{L\ref{lem:tree-emb}}:=K_j$, $R_{L\ref{lem:tree-emb}}:=K_j\cap N(R)$, $X'_{L\ref{lem:tree-emb}}:= B\setminus (U\cup W\cup \tilde W\cup \tilde U)$, $Y'_{L\ref{lem:tree-emb}}:= D\setminus (U\cup W\cup \tilde U\cup \tilde W)$, $\alpha_{L\ref{lem:tree-emb}}:=\frac{32\varepsilon}{r\eta}$,  $\varepsilon_{L\ref{lem:tree-emb}}:=\varepsilon$, and $d_{L\ref{lem:tree-emb}}:=d$ to extend~$\varphi $ to the whole~$K_j$ with $F_1\cap K_j\subseteq B$ and $F_2\cap K_j\subseteq D$ and add the used vertices to $\tilde W$. Observe that after the embedding of $K_j$, we still have in each cluster~$B$ and~$D$ at least $r\eta /8\cdot|B|$ and $r\eta /8\cdot|D|$ vertices, respectively, outside~$U$,~$W$,~$\tilde U$, and~$\tilde W$. We continue until every $K\in F-R$ is embedded.
\end{proof}

%misto $B$ dat $C$ do prop 17
%jak je definovano r?
%a_2 muze byt 0
%case B uz neni podmineny

\section{Proof of Proposition \ref{prop:mainembedding2}}\label{sec:main_prop}
%Proof of Proposition \ref{prop:mainembedding2}. 
Given $\delta, q, d>0$ and $\tilde r\le r'\le 1/2$ set \begin{align*}
\varepsilon&:= \min\left\{\varepsilon_{P\ref{prop:embed-SM}}(\frac{q\delta}{20}, d, r'),\varepsilon_{P\ref{prop:embed-LS}}(\frac{q\delta}{20}, d, r'), \left(\frac{\delta q}{3}\right)^2, d/17\right\},\\
\beta&:= \min\left\{\beta_{P\ref{prop:embed-SM}}(\frac{q\delta}{20}, r', \varepsilon, \tilde N_{max}), \beta_{P\ref{prop:embed-LS}}(\frac{q\delta}{20}, r', \varepsilon, \tilde N_{max}),\delta r'/8\right\},\\
n_0&:= \frac{200}{\delta qr' \beta}.
\end{align*}
%\marginpar{T:do $\varepsilon$ pridan pozadavek $<= d/17$ V: to dava smysl (i kdyz prvni podminka dava silnejsi odhad)}

We gradually construct an injective homomorphism~$\phi$ of~$T$ into~$H$. To this end we consider the four introduced cases. 

In each case, we start by embedding the vertices of~$W_A$ and~$W_B$
to ultratypical vertices of~$A$ and~$B$, respectively. 
This can be done by applying Lemma~\ref{lem:tree-emb} with $X'_{L\ref{lem:tree-emb}}$ and $Y'_{L\ref{lem:tree-emb}}$ being the sets of ultratypical vertices of~$A$ and~$B$, respectively, $T_{L\ref{lem:tree-emb}}$ being any tree with colour classes~$W_A$ and~$W_B$ such that $T[W_A\cup W_B]$ is a subgraph of  $T_{L\ref{lem:tree-emb}}$, $\alpha_{L\ref{lem:tree-emb}}=5\varepsilon$ and  $R_{L\ref{lem:tree-emb}}=\emptyset$. 
Note that the assumptions of Lemma~\ref{lem:tree-emb} are satisfied, since the pair $(A,B)$ has density at least $d-\varepsilon>15\varepsilon$, by Lemma~\ref{lem:ultratypical} at least $1-\sqrt{\varepsilon}>4/5$ of vertices of~$A$ or~$B$, respectively, are ultratypical, and moreover |$W_A|<\varepsilon |A|$, $|W_B|<\varepsilon |B|$ by definition of fine partition.

%\marginpar{T:zrusila jsem $\varepsilon'$}

We embed the rest of the tree~$T$ using different strategy for each case. In what follows, we use indexes $1$ and $2$ to denote that the structure is a substructure of~$T_1$ or~$T_2$, respectively.

When using Propositions~\ref{prop:embed-SM} and~\ref{prop:embed-LS}, we shall always use (here we use the index $_P$ to indicate the parameter of the propositions) $d_{P}:= d$, $r_P := r'$, $\tilde N_{max, P}:= \tilde N_{max}$, $n_P:=n$, $H_P:=H$,  $\mathbf H_P:=\mathbf H$, and $R_P$ will be either $W_A$ or $W_B$ depending whether we embed part of $\mathcal D_A$, or $\mathcal D_B$, respectively. 
In some cases, we shall use Proposition~\ref{prop:embed-LS} several times. To avoid confusion, we shall use upper indices in parenthesis, e.g., $U^{(1)}_{P\ref{prop:embed-LS}}$, to indicate to which application of the proposition we refer.
 We will write $\fD_{B1}$ as a shortcut for $\fD_B \cap V(T_1)$ and $\fD_{B2} := \fD_B \cap V(T_2)$ and $\fD_{A1}$ as a shortcut for $\fD_A \cap V(T_2)$ (sic) and $\fD_{A2} := \fD_A \cap V(T_1)$. Thus, neighbours of $W_A$ or $W_B$ are in $\fD_{A1}$ or $\fD_{A2}$, respectively.

\subsection*{Case A}
In this case we assume that there are two adjacent clusters~$A$ and~$B$ in~$H$ such that $\avdeg(A,\fS_1\cup \fS_M) \ge \frac{1-\tr}{\tr} \atwo + \delta k$ and $\avdeg(B,\mathcal L)\ge (\tr+\delta)k$. 
%as we already stated, we embed the seeds~$W_A$ and~$W_B$ in ultratypical vertices in~$A$ and~$B$, respectively. Then we continue  by embedding almost all vertices from~$\fD_A$ using the matching structure $\mathbf M$ and vertices from~$\mathcal \fS_1$. To this end we use Propositions~\ref{prop:embed-SM} and~\ref{prop:embed-LS}.  Then we embed the vertices from~$\fD_B$ through~$\mathcal L$ using Proposition \ref{prop:embed-LS}. Finally, we argue that we can easily extend $\phi$ to cover the left-out vertices of~$\fD_A$. 

We start by embedding the vertices of $W_A$ and $W_B$ to ultratypical vertices of clusters $A$ and $B$, respectively. We then further partition the rest of $T$ and embed it in the following three steps which we describe in detail later. We partition the trees from $\fD_A$ in two sets -- $\fF$ and $\fG$ and define $\fF'$ and $\fG'$ as sets of subtrees of $\fF$ and $\fG$, respectively, with leaves in $\fD_{A1}$ removed. We denote $\fF\cap \fD_{Ai}$ and $\fG\cap \fD_{Ai}$ by $\fF_i$ and $\fG_i$ respectively, for $i=1,2$. Analogously, we define $\fF'_i$ and $\fG'_i$ for $i=1,2$.

In the first step, we embed $\fF'$ into the edges of the matching $\mathbf{M}$ using Proposition \ref{prop:embed-SM} and we embed $\fG'$ through $\fS_1$ vertices using Proposition \ref{prop:embed-LS} (i.e., $\phi(\fG'_1) \subseteq \fS_1$ and $\phi(\fG'_2) \subseteq \fL$).  

In the second step, we embed the trees from $\fD_{B1}$ using again Proposition \ref{prop:embed-LS}. To this end we again use the bound on the degree of the cluster $B$ -- specifically, as $\avdeg(B, \fL) \ge \tilde rk + \delta k = |\fD_{A2} \cup \fD_{B1}| + \delta k$, the cluster $B$ has enough neighbours for embedding $\fD_{B1}$, even though $\fD_{A2}$ is already embedded. 
%We postpone embedding of leaf vertices of $\fF_1 \cup \fG_1$ and begin with embedding smaller sets $\fF' \subseteq \fF$ and $\fG' \subseteq \fG$ using propositions from previous section.

%The size of $\fF$ is chosen so that it can be embedded; the condition on the degree of the cluster $A$ then ensures that $\fG$ (i.e., the rest of $\fD_A$) is small enough to be embedded afterwards. 
In the third step, we embed $\fF\setminus \fF'$ and $\fG\setminus \fG'$ greedily. The structure of the embedded tree is sketched in Figure \ref{fig:caseA}.

\begin{figure}
    \centering
    \includegraphics{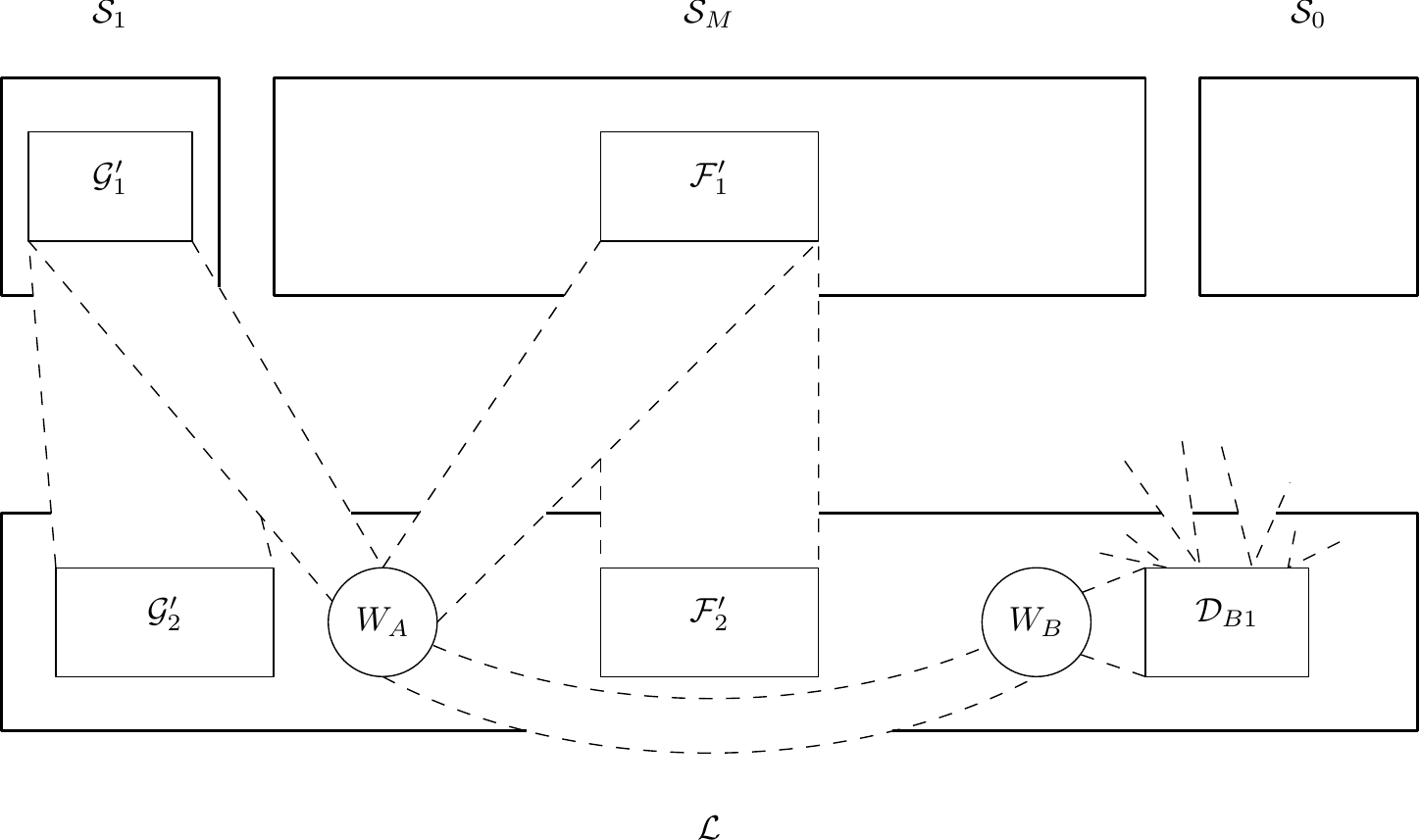}
    \caption{The embedding configuration in the case A. After inserting the vertices of $W_A, W_B$ in the ultratypical vertices of clusters $A$ and $B$ we use Proposition \ref{prop:embed-SM} to embed $\fF'$ in the matching connecting $\fS_M := \fS \cap M$ and $\fL \cap M$. Then we invoke Proposition \ref{prop:embed-LS} to embed $\fG'$ using the vertices in $\fS_1$. Finally, we again invoke Proposition \ref{prop:embed-LS} to embed $\fD_B$. 
    Note that in this case, as well as in all of the subsequent cases, it may be the case that $B \in \fS_1$. }
    \label{fig:caseA}
\end{figure}

\begin{enumerate}
    \item 
    In this step we embed the trees from the anchored forest $\fD_A$ except of several leaves, ensuring that the neighbours of those left-out leaves are mapped to ultratypical vertices in $\fL$-clusters. 
    We split the anchored~$\beta k$-forest~$\fD_A$ into two disjoint forests~$\fF$ and~$\fG$ in the following way. 
    Let~$\fF$ be a maximal subset of trees of~$\fD_A$ such that 
    \begin{align}
        \label{A:F_small}
        |\fF_2| \leq \frac{r'}{1-r'} \avdeg(A, \fS_M) - \frac{r'}{1-r'} \delta k/2,
    \end{align}
    and we choose it as an empty set if the size of the expression is less than zero. %Define $\fF_1 := \fF \cap V(T_2)$, $\fF_2 := \fF \cap V(T_1)$, and similarly $\fG_1 := \fG \cap V(T_2)$ and $\fG_2 := \fG \cap V(T_1)$. 
    
    This means that if $\fG$ is non-empty then 
    \begin{align}
        \label{A:F_large}
        |\fF_2| \geq \frac{r'}{1-r'} \avdeg(A, \fS_M) - \frac{r'}{1-r'} \delta k/2 - \beta k,
    \end{align}
    otherwise we could move a suitable tree from $\fG$ to $\fF$ while retaining the condition imposed on $\fF$. 
    By deleting the leaves of trees in $\fD_A$ that are contained in $\fF_1 \cup \fG_1$ we get forests $\fF'$ and $\fG'$. For each tree $K \in \fF' \cup \fG'$ we have $|K \cap (\fF'_1 \cup \fG'_1)| \leq |K \cap (\fF'_2 \cup \fG'_2)|$, because each vertex from $K \cap (\fF'_1 \cup \fG'_1)$ has at least one child in $(\fF'_2 \cup \fG'_2)$. Specifically, $|\fG'_1| \leq |\fG'_2|$. 
    
    Now we apply Proposition~\ref{prop:embed-SM} to our anchored forest $F_{P\ref{prop:embed-SM}}:=\fF'$ if it is non-empty. Set  $U_\psm := \phi(W_A \cup W_B)$,  $\eta_\psm := q \delta /4$, $\mathbf M_\psm:=\mathbf M$, and $A_\psm:=A$. From Definition \ref{def:partition} we know that $|U_\psm| = |W_A \cup W_B| \leq 12k / (\beta k) = 12/\beta$.

    To apply the proposition it suffices to verify that the degree of $A$ in $\fS_M$ is sufficiently large, as by definition of $\fF'$ we know that for each $K \in \fD_A$ we have $|K \cap \fF'_1| \leq |K \cap \fF'_2|$. We have
    
    \begin{align*}
        \avdeg(A,\mathcal \fS_M) 
        &\geq \frac{1-r'}{r'} |\fF_2| + \delta k/2\\
        &\geq \frac{1-r'}{r'} |\fF_2| + \frac{1-r'}{r'} |U_\psm| + \delta k/4\\
        &\geq \frac{1-r'}{r'} |\fF'_2| + \sum_{C\subseteq  \fS\::\: CD\in \mathbf M}\max\{|U_\psm\cap C|, \frac{1-r'}{r'}|U_\psm\cap D|\} \; + \eta_\psm n\;,
    \end{align*}
    where the first inequality is due to the definition of $\fF$ (bound \ref{A:F_small}) and the second one and third one are due to the facts that $\delta k / 4 \geq  \frac{1-r'}{r'} |U_\psm|$ (from the choice of $n_0$) and $\delta k /4 \geq \eta_\psm n$ (from the choice of $\eta_{P\ref{prop:embed-SM}}$).

    If $\fG$ is non-empty and, thus, the bound \ref{A:F_large} holds, we proceed by embedding $\fG'$. 
    
    We apply Proposition \ref{prop:embed-LS} (Configuration 2) to the anchored forest~$\fG'$ and $\mathcal B_\pls^{(1)}:=\mathcal S_1$. As we know that~$N_{\mathbf H}(\mathcal S_1)$ is disjoint from~$\mathcal S_M$, there is no need to include $\phi(\fF_1) \subseteq \fS_M$ in the forbidden set~$U$ that ensures the injectiveness of~$\phi$. 
    Thus, we set $U_\pls^{(1)}:=\phi(\fF_2 \cup W_A \cup W_B)$. 
    
    Also note that $\bigcup \mathcal B_\pls^{(1)} \cap U_\pls^{(1)} \subseteq \varphi(W_A \cup W_B)$, because $\phi(\fF_2) \in \fL$ (we could actually replace $W_A \cup W_B$ by $W_B$). 
    Let $\eta_\pls^{(1)} := \delta q/4$, and $A_\pls^{(1)}:=A$.
    Now we verify the first condition from Proposition~\ref{prop:embed-LS}. 
    For the degree of the cluster $A$ in $\mathcal S_1$ we have
    \begin{align*}
        \avdeg(A,\mathcal S_1)
        & = \avdeg(A, \mathcal S_1 \cup\mathcal  S_M) - \avdeg(A,\mathcal  S_M)\\
        \just{assumption of this configuration} &\geq \frac{1-\tr}{\tr} \atwo + \delta k - \avdeg(A,\mathcal  S_M)\\
        \just{$\tr \le r'$, bound \eqref{A:F_large}} & \geq \frac{1-r'}{r'} |\fF_2 \cup \fG_2| + \delta k - \frac{1-r'}{r'} |\fF_2| - \delta k/2  - \frac{1-r'}{r'} \beta k\\
        \just{bounding error terms} & \geq \frac{1-r'}{r'} |\fG_2| + 3\delta k/8\\
        \just{$|\fG_2| \ge |\fG'_2| \ge |\fG'_1|$} &\geq |\fG'_2| + 3\delta k /8\\
        &\geq |\fG'_1| + |\bigcup \mathcal B_\pls^{(1)} \cap U_\pls^{(1)}| + \eta_\pls^{(1)} n\;, 
    \end{align*}
    where we at first used the lower bound on the degree of $A$ in $\mathcal S_1 \cup \mathcal S_M$, then the lower bound on the size of $\fF_2$, after bounding the error terms we used the fact that $|\fG'_2| \geq |\fG'_1|$ and then we again bounded the errors terms by using the facts that $|\bigcup \mathcal B_\pls^{(1)} \cap U_\pls^{(1)}| \leq |W_A \cup W_B| \leq \delta k/8$ and $\eta_\pls^{(1)} n \leq \delta k /4$. 
    
    Further we verify that for each cluster $C \in\mathcal  S_1$ we have
    \begin{align*}
        \avdeg(C, V(\mathbf H) \setminus \mathcal B_\pls^{(1)})=\avdeg(C,\mathcal L) 
         &\geq \tr k+\delta k\\ 
       \just{bound on the skew of $T$} &\geq \atwo + |\phi(W_A \cup W_B)| + \eta_\pls^{(1)} n \\
        &= |\fF_2| + |\fG_2| + |\phi(W_A \cup W_B)| + \eta_\pls^{(1)} n\\
        &\geq |\fG'_2| + |U_\pls^{(1)}| + \eta_\pls^{(1)} n\;,
    \end{align*}
    %where we used the facts that $\atwo \le  rk$ and bounded the error terms as usual. 
    Thus we can extend $\phi$ to $\fG$. Note that $\phi(\fG_2) \subseteq \fL$. 

    % In subsequent applications of Propositions \ref{prop:embed-LS} we do not define $R_\pls$, as it is always set to either~$W_A$ when we embed trees from~$\fD_A$ or~$W_B$ when we embed the trees from~$\fD_B$.     
    \item
    In this step we embed the trees from $\fD_B$ using Configuration 1 from Proposition \ref{prop:embed-LS}.
    %For simplicity of notation, we use the symbols $\fH_1, \fH_2$ to denote the two colour classes of vertices of $\fD_B$ as in the previous cases. 
    The appropriate set $U_\pls^{(2)}$ guaranteeing the injectiveness of $\phi$ consists of $\phi(\fF'_1 \cup \fF'_2 \cup \fG'_1 \cup \fG'_2 \cup W_A \cup W_B)$. We set $\mathcal B_\pls^{(2)}:=\mathcal L$, $\eta_\pls^{(2)} := \delta q/2$ and $A_\pls^{(2)}:=B$. 
    First we verify the first condition of the proposition. We have
    \begin{align*}
        \avdeg(B,\mathcal L)
        &\geq \tr k + \delta k\\
       \just{bound on the skew of $T$} &= |\fF_2 \cup \fG_2 \cup \fD_{B1}| + \delta k\\
       \just{bounding error terms} &\geq |\phi(\fF_2 \cup \fG_2 \cup W_A \cup W_B)| + |\fD_{B1}| + \delta k/2\\
        &\geq |\bigcup \mathcal B^{(2)}_\pls \cap U_\pls^{(2)}| + |\fD_{B1}| + \eta_\pls^{(2)} n\;,
    \end{align*}
    %where we use the fact that $rk \ge  \atwo + \bone$ and then bound the error terms. 

    We immediately use the 'moreover' part of the proposition with $\tilde U_\pls^{(2)} = \emptyset$ and verify that for each $L$-cluster $C$ we have
    \begin{align*}
        \avdeg(C)
        &\geq k+\delta k\\
        &\geq |\fD_{B1}|+|\fD_{B2}|+|\fD_A|+|W_A \cup W_B|+\delta k\\
        &\geq |\fD_{B1}|+|\fD_{B2}|+|U_\pls^{(2)}|+\eta_\pls^{(2)} n\;,\\
    \end{align*}
    where we use mainly the fact that $|\fD_A \cup \fD_B \cup W_A \cup W_B|=k$. 
    \item
    We have defined an injective homomorphism $\phi$ of the whole tree $T$ except of its leaves from $\fF_1 \setminus \fF_1'$ and $\fG_1 \setminus \fG'_1$. We know that their neighbours are embedded in ultratypical vertices of $L$-clusters. By Proposition~\ref{prop:degrutratypical}, such vertices have degree at least $k+\delta k-2\sqrt{\varepsilon}n/r' \ge k$ as $\delta q>2\sqrt{\varepsilon}/r'$. Thus we can greedily extend $\phi$ to the whole tree $T$.% in the following greedy manner. We iterate over all non-embedded leaves $v \in (\fF_1 \setminus \fF_1') \cup (\fG_1 \setminus \fG'_1)$. As we know that $\deg_G(\phi(N(v))) \ge  |V(T)|$, there is always a neighbour of $\phi(N(v))$ outside of the partial embedding $\phi(T)$ that can be used for extending $\phi$ to $u$. 
\end{enumerate}

\subsection*{Case B}
    In this case we assume that $\tr \aone \geq (1-\tr) \atwo$ and that there are two adjacent clusters $A$, $B$ such that $\avdeg(A,\mathcal S_1\cup \mathcal S_M\cup \mathcal L)\ge (1+\delta)k$ and $\avdeg(B, \mathcal L)\ge (\tr+\delta)k$. 
    The embedding procedure is roughly similar to the one from Case~A. However, for embedding~$\fD_A$ we now also use~$\mathcal L$.% apart from~$\mathcal S_M$ and~$\mathcal S_1$. 
    
    We start by embedding certain part of the anchored forest~$\fD_A$ using the matching~$\mathbf M$ and the set~$\mathcal S_1$ similarly to the Case~A. Then, we proceed by reserving~$|\fD_{B1}|$ vertices that will later help us to embed the anchored trees from~$\fD_B$. In the third part we embed the rest of the forest~$\fD_A$ using the high degree vertices in~$\mathcal L$, and then proceed by embedding~$\fD_B$ using the reserved vertices. Finally, we argue that we can embed several leftover leaves of the tree as in the previous case. 

\begin{figure}
    \centering
    \includegraphics{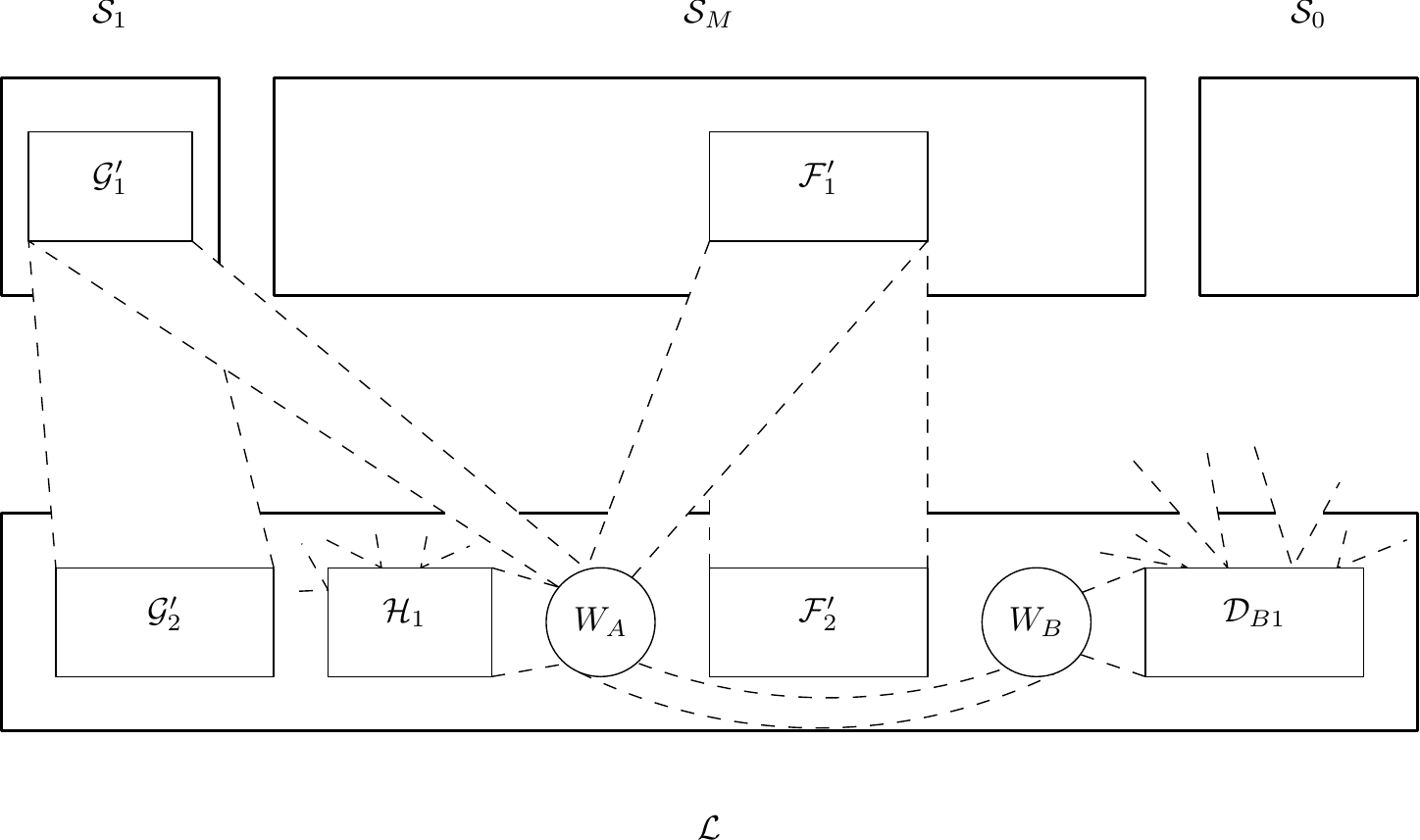}
    \caption{The embedding configuration in the case B. After inserting the vertices of $W_A, W_B$ in the ultratypical vertices of clusters $A$ and $B$ we use Proposition \ref{prop:embed-SM} to embed $\fF'$ in the matching $\mathbf{M}$. Then we invoke Proposition \ref{prop:embed-LS} to embed $\fG'$ using the vertices in $\fS_1$. Then we reserve suitable vertices in the neighbourhood of the cluster $B$ in $\fL$ that will later serve for embedding of $\fD_{B1}$ using Proposition \ref{prop:embed-LS}. Then we embed $\fH_1$ using the same proposition and finally we embed $\fD_{B}$ through the reserved vertices. }
    \label{fig:caseB}
\end{figure}
	
\begin{enumerate}
    \item 
    Analogously to the preceding case we split the anchored forest $\fD_A$ into three disjoint sets $\fF = \fF_1 \cup \fF_2$, $\fG = \fG_1 \cup \fG_2$, and $\fH = \fH_1 \cup \fH_2$ in the following way. 
    
    Let $K_1, K_2, \dots$ be the trees of $\fD_A$ sorted according to their skew, i.e., according to the ratio $|K_i \cap V(T_2)| / |K_i \cap V(T_1)|$ in descending order. 
    We define $\fF$ as the union $K_1 \cup \dots \cup K_j$, where $j$ is taken to be maximal such that 
    \begin{align}
        |\fF_2| = \sum_{i=1}^j |K_i \cap V(T_1)| \leq \frac{r'}{1-r'} \avdeg(A,\mathcal  S_M) - \frac{r'}{1-r'} \delta k/3.      
    \end{align}
    If the right hand side is less than zero, define~$\fF$ as the empty set. 
    Then we similarly define~$\fG$ as the union of trees $K_{j+1}, \dots, K_{j'}$ where $j'$ is maximal such that 
    \begin{align}
        |\fG_2| = \sum_{i=j+1}^{j'} |K_i \cap V(T_1)| \leq \frac{r'}{1-r'}\avdeg(A,\mathcal  S_1)-\frac{r'}{1-r'}\delta k/3. 
    \end{align}
    Finally we set $\fH=\fD_A \setminus(\fF \cup \fG)$. 

    As before, we have 
    \begin{align}
        \label{B:F_large}
        |\fF_2| 
        \geq \frac{r'}{1-r'} \avdeg(A,\mathcal  S_M) - \frac{r'}{1-r'} \delta k/3 - \beta k \;,  
    \end{align}
    if $\fF \not = \fD_A$ and
    \begin{align}
        \label{B:G_large}
        |\fG_2| 
        \geq \frac{r'}{1-r'}\avdeg(A,\mathcal  S_1)-\frac{r'}{1-r'}\delta k/3 - \beta k \;,
    \end{align}
    if $\fF \cup \fG \not = \fD_A$.
    Additionally, we also have 
    \begin{align}
        \label{B:FG_skew}
        \tr |\fF_1 \cup \fG_1| \geq (1-\tr) |\fF_2 \cup \fG_2|   \;,
    \end{align}
    because of the assumption $\tr \aone \geq (1-\tr) \atwo$ and the fact that in $\fF \cup \fG$ there are the anchored trees with biggest skew.
    
    %mozna to se skew neplati, do a_2 se asi pocitaji y
    
    We define~$\fF'$ and~$\fG'$ as in the previous case. We have $|K \cap \fF'_1| \leq |K \cap \fF'_2|$ for each $K \in \fF'$ and $|\fG'_1| \leq |\fG'_2|$. 
    
    If $\fF'$ is non-empty we apply Proposition \ref{prop:embed-SM} to embed the anchored forest~$F_\psm:=\fF'$ in the same way as in the previous case. Set $U_\psm = \phi(W_A \cup W_B)$, $\eta_\psm=\delta q / 4$, $r_\psm:= r'$,  $\mathbf M_\psm:=\mathbf M$, and $A_\psm:=A$.
    Similarly to the previous case we verify that
    \begin{align*}
        \avdeg(A,\mathcal S_M) 
        &\geq \frac{1-r'}{r'} |\fF_2| + \delta k/3\\
        &\geq \frac{1-r'}{r'} |\fF_2| + \frac{1-r'}{r'} |U_\psm| + \delta k/4\\
        &\geq \frac{1-r'}{r'} |\fF'_2| + \sum_{C\subseteq  \fS\::\: CD\in \mathbf M}\max\{|U_\psm\cap C|, \frac{1-r'}{r'}|U_\psm\cap D|\} \; + \eta_\psm n\;.
    \end{align*}
    
    If $\fG$ is non-empty we proceed by embedding $\fG'$. This is also done in an analogous way to the preceding case. 
    
    We apply Proposition \ref{prop:embed-LS} (Configuration 2) to the anchored forest $F_\pls^{(1)}:=\fG'$ and set $\mathcal B_\pls^{(1)}:=\mathcal S_1$. 
     By the properties of a skew LKS graph, the set  $N_{\mathbf H}(\mathcal S_1) \cup \mathcal S_1$ is disjoint from $\mathcal S_M \supseteq \phi(\fF'_1)$, thus for ensuring injectiveness of~$\phi$ it suffices to set $U_\pls^{(1)} := \phi(\fF'_2 \cup W_A \cup W_B)$ and then we also have $ \bigcup \mathcal B_\pls^{(1)} \cap U_\pls^{(1)} \subseteq W_A \cup W_B$. 
    Set $\eta_\pls^{(1)} := \delta q /4$, and $A_\pls^{(1)}:=A$.
    
    Now we verify the first condition from the proposition. For the degree of the cluster~$A$ in~$\mathcal S_1$ we have
    \begin{align*}
        \avdeg(A,\mathcal S_1) 
        &\geq \frac{1-r'}{r'} |\fG_2| + \delta k/3\\
        &\geq |\fG'_1| + |\bigcup \mathcal B_\pls^{(1)} \cap U_\pls^{(1)}| + \eta_\pls^{(1)} n\;, 
    \end{align*}
    where we use the definition of $\fG$, the fact that $|\fG_2| \geq |\fG'_1|$ and the fact that $|\bigcup \mathcal B_\pls^{(1)} \cap U_\pls^{(1)}| \leq \delta q / 12$.
    
    Further, we verify that for each cluster $C \in \fS_1$ we have
    \begin{align*}
        \avdeg(C, V(\mathbf H) \setminus \mathcal B_\pls^{(1)})=\avdeg(C,\mathcal L) 
        &\geq \tr k+\delta k\\ 
        &\geq \atwo + |\phi(W_A \cup W_B)| + \delta k/2 \\
        &\geq |\fF_2| + |\fG_2| + |\phi(W_A \cup W_B)| + \delta k / 2\\
        &\geq |\fG'_2| + |U_\pls^{(1)}| + \eta_\pls^{(1)} n\;,
    \end{align*}
    where we used the facts that $\atwo \leq \tr k$ and bounded the error terms in the usual manner.  

    \item
    In this step we reserve suitable vertices for embedding $\fD_B$ and use Proposition \ref{prop:embed-LS}, Configuration 1, to this end. 

    We apply the proposition the anchored forest $F_\pls^{(2)}:=\fD_B$, $A_\pls^{(2)}:=B$,  the set $U_\pls^{(2)}:=\phi(W_A \cup W_B \cup \fF_2 \cup \fG_2)$, and $\mathcal B_\pls^{(2)}:=\mathcal L$. Take $\eta_\pls^{(2)} := q\delta / 20$. We start by verifying the first condition:
    \begin{align*}
        \avdeg(B,\mathcal L)
        &\geq \tr k + \delta k\\
        &\geq |\fF_2 \cup \fG_2 \cup \fH_2 \cup \fD_{B1}| + \delta k\\
        &\geq |\fF_2 \cup \fG_2 \cup \fD_{B1}| + \delta k\\
        &\geq |\fD_{B1}| + |\phi(W_A \cup W_B \cup \fF_2 \cup \fG_2)| + \delta k/2\\
        &\geq |\fD_{B1}| + |U_\pls^{(2)}| + \eta_\pls^{(2)} n\;,\\
    \end{align*}
    where we use the upper bound on the smaller colour class of $T$ and then we bound the error terms as usual. This gives us an embedding of $N(W_B )\cap \fD_B$ as well as the reservation set~$W$ that will help us later for embedding~$\fD_{B1}$. 
      
    Before finishing the embedding of~$\fD_B$ by invoking the 'moreover' part of Proposition~\ref{prop:embed-LS}, Configuration~1, we shall embed the anchored forest~$\fH$, which will define the set $\tilde U_\pls^{(2)} := \phi(\fH)$. 
      
    \item
    We proceed with embedding of $F_\pls^{(3)}:=\fH$, using a third time Proposition~\ref{prop:embed-LS}, Configuration~1. Let $U'=\phi(N(W_B)\cap \fD_B) \cup W$, $|U'| = |\fD_{B1}|$ and set $U_\pls^{(3)} := \phi(W_A \cup W_B \cup \fF \cup \fG)\cup U'$. Thus $U_\pls^{(3)} \cap \fL \subseteq \phi(W_A \cup W_B \cup \fF'_2 \cup \fG'_2) \cup U'$. 
    Further set $\mathcal B_\pls^{(3)}:=\mathcal L$, $\eta_\pls^{(3)} := \delta q / 4 \ge \eta_\pls^{(2)}$, and $A_\pls^{(3)}:=A$.
    We verify the first condition of the proposition:
    \begin{align*}
        \avdeg(A,\mathcal  L)
        &\geq k + \delta k - \avdeg(A,\mathcal  S_M) - \avdeg(A, \mathcal S_1) \\
        \just{bounds \eqref{B:F_large} and \eqref{B:G_large} } & \geq k + \delta k - (\frac{1-r'}{r'} |\fF_2| + \delta k/3 + \frac{1-r'}{r'} \beta k) - (\frac{1-r'}{r'} |\fG_2| + \delta k/3 + \frac{1-r'}{r'} \beta k)\\
        \just{bounding error terms} & \geq k - \frac{1-r'}{r'} (|\fF_2| + |\fG_2|) + \delta k/4 \\
        \just{$\tilde r\le r'$}&\ge k - \frac{1-\tr}{\tr} (|\fF_2| + |\fG_2|) + \delta k/4 \\
        \just{bound \eqref{B:FG_skew} } & \geq k - (|\fF_1| + |\fG_1|) + \delta k/4\\
        \just{$T$ is of size $k$} & \geq |\fF_2|+|\fG_2|+|\fH|+|\fD_B|+|W_A \cup W_B| + \delta k/4\\
        &\geq |\fH_1| + |\phi(W_A \cup W_B \cup \fF'_2 \cup \fG'_2)| + |\fD_{B1}| + \delta k/4\\
        &\geq |\fH_1| + |U_\pls^{(3)} \cap \bigcup \mathcal B^{(3)}_\pls| + \eta_\pls^{(3)} n\;,
    \end{align*}
    where we at first used our bounds on $|\fF_2|$ and $|\fG_2|$. Then we used the inequality  $\tr |\fF_1 \cup \fG_1| \geq (1-\tr) |\fF_2 \cup \fG_2|$, we followed by interpreting $k$ as the size of $T$ and used trivial bounds on error term throughout the computation. 

    We immediately use the second part of the proposition with $\tilde U_\pls^{(3)} = \emptyset$. We verify that for each $C \in \mathcal L$ we have
    \begin{align*}
        \avdeg(C)
        &\geq k + \delta k\\
        &=|\fF \cup \fG \cup \fH \cup \fD_B \cup W_A \cup W_B| + \delta k\\
        &\geq |\fH_1| + |\fH_2| + |U_\pls^{(3)} \cup \tilde U_\pls^{(3)}| + \eta_\pls^{(3)} n\;.
    \end{align*}

    Thus, we can extend $\phi$ to $\fH$. Note that $|C \setminus (U \cup U' \cup \tilde U)| \ge r' \eta_\pls^{(3)}|C|/8$ for each cluster~$C$ with $C \cap \phi(\fH)$.

    \item 
    Now, we finish up the embedding of~$\fD_B$, using the 'moreover' part of the second application of Proposition~\ref{prop:embed-LS}. The first condition of the proposition is satisfied, as for each $C \in\mathcal  L$ we have
    \begin{align*}
        \avdeg(C)
        &\geq k + \delta k\\
        &= |\fD_A \cup \fD_B \cup W_A \cup W_B| + \delta k\\
        &\geq |\fD_{B1}| + |\fD_{B2}| + |\phi(W_A \cup W_B \cup \fF' \cup \fG' \cup \fH)| + \eta_\pls^{(2)} n\\
        &\geq |\fD_{B1}| + |\fD_{B2}| + |U_\pls^{(2)} \cup \tilde U_\pls^{(2)}| + \eta_\pls^{(2)} n\;.
    \end{align*}
    The second condition is that for each cluster~$C$ with $C \cap \phi(\fH)$ we have $|C \setminus (U \cup U' \cup \tilde U)| \ge r' \eta_\pls^{(2)}|C|/8$. This is satisfied as $\eta_\pls^{(3)} \geq \eta_\pls^{(2)}$ and by the property of the embedding of~$\fH$, guaranteed by the third application of Proposition \ref{prop:embed-LS}.

    \item   
    We have defined an injective homomorphism $\phi$ on the whole tree $T$ except of its leaves from $\fF_1 \setminus \fF'_1$ and $\fG_1 \setminus \fG'_1$. As we know that their neighbours are embedded in ultratypical vertices of $L$-clusters, we can greedily extend the embedding to the whole tree~$T$, as in Case~A. 
\end{enumerate}

\subsection*{Case C}
    In this case we assume that $\tr \aone \leq (1-\tr) \atwo$ and that there are adjacent clusters~$A$ and~$B$ such that $\avdeg(A,\mathcal S_1\cup\mathcal  S_M\cup\mathcal  L)\ge (1+\delta)k$ and $\avdeg(B,\mathcal  L)\ge \bone+\delta k = |\fD_{B1}|+\delta k$. The embedding procedure is very similar to the one from the preceding case, the difference being in the order in which we embed the parts of~$T$ in the host graph. 
    
    We start by reserving vertices for the embedding of the anchored forest~$\fD_B$ using Proposition~\ref{prop:embed-LS}. Then we embed parts of~$\fD_A$ using the matching~$\mathbf M$ and~$\mathcal S_1$ as in the previous cases. We have to be more careful, though, as the vertices reserved for~$\fD_B$ can cover substantial part of~$\mathbf M$. We finish by embedding the rest of~$\fD_A$ through high degree $L$-clusters using Proposition~\ref{prop:embed-LS}.

    \begin{figure}
        \centering
        \includegraphics{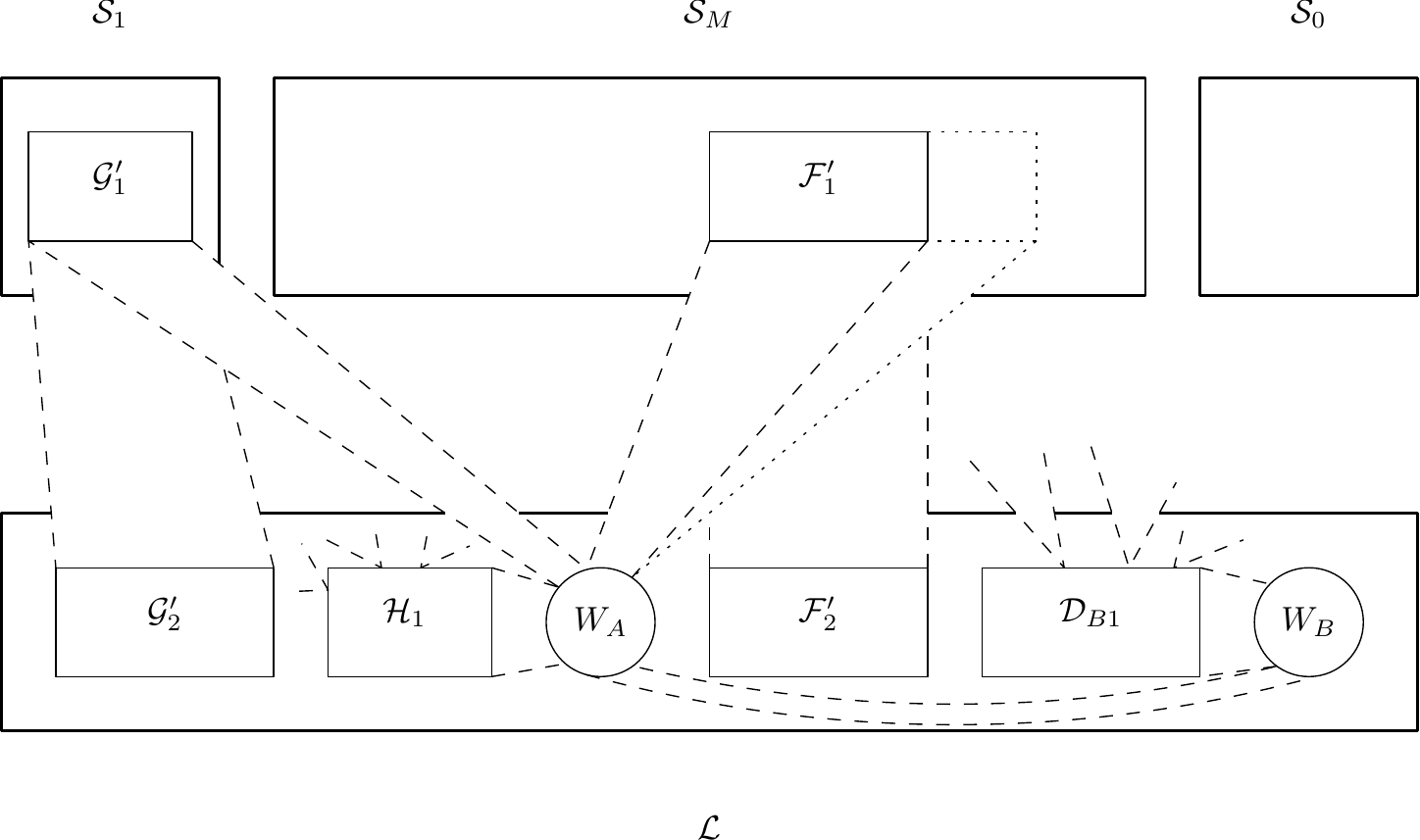}
        \caption{The embedding configuration in the case C. The configuration is very similar to the preceding one from case B. However, in this case we start by embedding $\fD_B$ in the neighbourhood of the cluster $B$. The figure suggests that because of the vertices reserved for $\fD_{B1}$ we must be more careful in the application of Proposition \ref{prop:embed-SM} and add those vertices in the forbidden set $U_{P\ref{prop:embed-SM}}$. }
        \label{fig:caseC}
    \end{figure}
    
    \begin{enumerate}

    \item
    We start by reserving vertices for embedding the anchored forest $\fD_B$ such that $\fD_{B1} := \fD_B \cap V(T_1)$ will be embedded in the neighbourhood of the cluster $B$. Set $\mathcal B_\pls^{(1)} :=\mathcal  L$ and $U_\pls^{(1)} := \phi(W_A \cup W_B)$.
    Set $\eta_\pls^{(1)} := q\delta / 20$, and $A_\pls^{(1)}:=B$. 
    We apply Proposition \ref{prop:embed-LS}, Configuration 1, to reserve vertices in~$\fL$ that will later serve for embedding of $\fD_B$. We verify that the first condition of the proposition is satisfied. Indeed:
    \begin{align*}
        \avdeg(B,\mathcal  L)
        &\geq |\fD_{B1}| + \delta k\\
        &\geq |\fD_{B1}| + |\phi(W_A \cup W_B)| + \eta_\pls^{(1)} n\;,
    \end{align*}
    where we used the standard error estimation. 
    
    This gives us embedding of $N(W_B) \cap \fD_B$ as well as a reserved set~$W$. We set $U'=\phi(N(W_B) \cap \fD_B) \cup W$, $|U'| = |\fD_{B1}|$. After embedding the whole~$T$ except of several of its leaf neighbours, we will invoke the second part of the proposition with $\tilde U^{(1)} = \phi(\fF' \cup \fG' \cup \fH)$ where $\fF' \cup \fG' \cup \fH \subseteq \fD_A$. 
    Note that if we set~$\tilde U^{(1)}$ to such value, we will satisfy the first condition needed for the actual embedding of~$\fD_B$, because for any cluster $C \in \mathcal L$ we have
    \begin{align*}
        \avdeg(C)
        &\geq k + \delta k\\
        &=|\fD_A \cup \fD_B \cup W_A \cup W_B| + \delta k\\
        &\geq |\fD_{B1}| + |\fD_{B2}| + |U_\pls^{(1)} \cup \tilde U^{(1)}| + \eta_\pls^{(1)} n\;.
    \end{align*}
    To satisfy the second condition we will ensure that for all subsequent applications of Propositions~\ref{prop:embed-SM} and~\ref{prop:embed-LS} we choose the value~$\eta$ being greater than~$\eta_\pls^{(1)}$. 
    
    \item
    
    We now proceed by embedding the anchored forest $\fD_A$ analogously to the previous case. We split the forest $\fD_A$ into three (possibly empty) forests $\fF, \fG, \fH$ in such a way that $\fF$ is maximal with
    \begin{align}
        |\fF_2| \leq \frac{r'}{1-r'} \avdeg(A, \fS_M) - |U'| - \frac{r'}{1-r'} \delta k / 3\;,
    \end{align}
    or $\fF$ is empty if the value of right hand side is smaller then zero. Moreover, if $\fF \not = \fD_A$, we have
    \begin{align}
        \label{C:F_big}
        |\fF_2| \geq \frac{r'}{1-r'} \avdeg(A, \fS_M) - |U'| - \frac{r'}{1-r'} \delta k / 3 - \beta k\;.
    \end{align}
    Then we similarly define $\fG$ to be maximal such that
    \begin{align}
        |\fG_2| \leq \frac{r'}{1-r'} \avdeg(A, \fS_1) - \frac{r'}{1-r'} \delta k / 3\;,
    \end{align}
    or $\fG$ is empty if the value of right hand side is smaller then zero. Moreover, if $\fF \cup \fG \not = \fD_A$, we have
    \begin{align}
        \label{C:G_big}
        |\fG_2| \geq \frac{r'}{1-r'} \avdeg(A, \fS_1) - \frac{r'}{1-r'} \delta k / 3 - \beta k\;.
    \end{align}
    
    We have $\fH := \fD_A \setminus (\fF \cup \fG)$ and, as in the previous case, $\fF \cup \fG$ consist of the trees with big skew, so if $\fD_{A2}$ is non-empty we have:
    \begin{align}
        \label{C:H_skew}
        \frac{1-\tr}{\tr} \geq 
        \frac{|\fD_{A1}|}{|\fD_{A2}|} \geq
        \frac{|\fH_1|}{|\fH_2|}\;. 
    \end{align}

    We define $\fF'$ and $\fG'$ as usual. 
    We use Proposition~\ref{prop:embed-SM} to embed the forest~$F_{P\ref{prop:embed-SM}}:=\fF'$ as in the previous cases. 
    Set $U_\psm := \phi(W_A \cup W_B) \cup U'$, $\eta_\psm := \delta q / 4$, $\mathbf M_\psm:=\mathbf M$, and $A_\psm:=A$. We verify that
    \begin{align*}
        \avdeg(A,\mathcal S_M) 
        &\geq \frac{1-r'}{r'} |\fF_2| + \frac{1-r'}{r'} |U'|+ \delta k/3\\
        &\geq \frac{1-r'}{r'} |\fF_2| + \frac{1-r'}{r'} |U_\psm| + \delta k/4\\
        &\geq \frac{1-r'}{r'} |\fF'_2| + \sum_{C\subseteq  \fS\::\: CD\in \mathbf M}\max\{|U_\psm\cap C|, \frac{1-r'}{r'}|U_\psm\cap D|\} \; + \eta_\psm n\;,
    \end{align*}
    where we used the fact that $|U_\psm| = |\phi(W_A \cup W_B)| + |U'| \leq |U'| + \delta k/12$. 
    
    If $\fG$ is non-empty, we proceed by embedding $\fG'$. As in the preceding cases, we apply Proposition~\ref{prop:embed-LS}, Configuration 2, to $F_\pls^{(2)}:=\fG'$ and set $\mathcal B_\pls^{(2)} := \mathcal S_1$. 
    As we know that $N_{\mathbf{H}}(\mathcal \fS_1) \cup \mathcal \fS_1$ is disjoint from $\bigcup \mathcal \fS_M \supseteq \phi(\fF'_1)$, for ensuring the injectiveness of $\phi$ it suffices to set $U_\pls^{(2)} := \phi(\fF'_2 \cup W_A \cup W_B) \cup U'$. 
    Because $\phi(\fF'_2) \cup U' \subseteq \fL$, we have $\bigcup \mathcal B_\pls^{(2)} \cap U_\pls^{(2)} \subseteq \phi(W_A \cup W_B)$. Set $\eta_\pls^{(2)} := \delta q / 4$, and $A_\pls^{(2)}:=A$. 
    We start by verifying the first condition from the proposition. We have
    \begin{align*}
        \avdeg(A,\mathcal \fS_1) 
        &\geq \frac{1-r'}{r'} |\fG_2| + \delta k/3\\
        &\geq |\fG'_1| + |\bigcup \fB_\pls^{(2)} \cap U_\pls^{(2)}| + \eta_\pls^{(2)} n\;, 
    \end{align*}
    where we use the definition of $\fG$, the fact that $|\fG_2| \geq |\fG'_1|$ and the fact that $|\bigcup \mathcal B_\pls^{(2)} \cap U_\pls^{(2)}| \leq 12/\beta$. 
    
    Further we verify that for each cluster $C \in \mathcal \fS_1$ we have
    \begin{align*}
        \avdeg(C, V(\mathbf H) \setminus \bigcup \mathcal B_\pls^{(2)})=\avdeg(C,\mathcal L) 
        &\geq \tr k+\delta k\\ 
        \just{bound on skew of $T$} & \geq \atwo + \bone + \delta k\\
        &\geq (|\fF_2| + |\fG_2|) + |\fD_{B1}| + |\phi(W_A \cup W_B)| + \delta k/2\\
        &\geq |\fG'_2| + |\phi(\fF'_2 \cup W_A \cup W_B)| + |U'| + \delta k/2\\
        &\geq |\fG'_2| + |U_\pls^{(2)}| + \eta n\;,
    \end{align*}
    where we started by using the bound on the skew of $T$, i.e., $\atwo + \bone \leq \tr k$, then bounded the error terms and rearranged suitable terms. 
    
    \item
    Now we apply Proposition~\ref{prop:embed-LS}, the first part, to embed the forest $F_\pls^{(3)}:=\fH$. Set $\mathcal B_\pls^{(3)} :=\mathcal  L$ and $U_\pls^{(3)} := \phi(W_A \cup W_B \cup \fF' \cup \fG') \cup U'$, thus $U_\pls^{(3)} \cap \fL \subseteq \phi(W_A \cup W_B \cup \fF'_2 \cup \fG'_2)\cup U'$. Set $\eta_\pls^{(3)}:=\delta q/8$, and $A_\pls^{(3)}:=A$. We start by verifying the first condition:
    \begin{align*}
        \avdeg(A,\mathcal  L) 
        &\geq k + \delta k - \avdeg(A, \fS_M) - \avdeg(A, \fS_1)\\
        \just{bounds \eqref{C:F_big} and \eqref{C:G_big}} & \geq k + \delta k - (\frac{1-r'}{r'} |\fF_2| + \delta k/3 + \frac{1-r'}{r'} |\fD_{B1}| + \frac{1-r'}{r'} \beta k) \\ 
        &\;\;\;\; - (\frac{1-r'}{r'} |\fG_2| + \delta k/3 + \frac{1-r'}{r'} \beta k)\\
        \just{bounding error terms} & \geq k - \frac{1-r'}{r'} (|\fF_2 \cup \fG_2 \cup \fD_{B1}|) + \delta k/4 \\
        \just{$\tilde r\le r'$}& \geq k - \frac{1-\tr}{\tr} (|\fF_2 \cup \fG_2 \cup \fD_{B1}|) + \delta k/4 \\
        &= k - \frac{1}{\tr}(|\fF_2 \cup \fG_2 \cup \fD_{B1}|) + (|\fF_2 \cup \fG_2 \cup \fD_{B1}|) + \delta k/4\\
        \just{bound on skew of $T$} & \geq k - \frac{1}{\tr} (\tr k - |\fH_2|) + (|\fF_2 \cup \fG_2 \cup \fD_{B1}|) + \delta k/4\\
        &= \frac{1}{\tr} |\fH_2| + (|\fF_2 \cup \fG_2 \cup \fD_{B1}|) + \delta k/4\\
        &\geq \frac{1-\tr}{\tr} |\fH_2| + (|\fF_2 \cup \fG_2 \cup \fD_{B1}|) + \delta k/4\\
        \just{bound \eqref{C:H_skew}} & \geq |\fH_1| + |\phi(W_A \cup W_B \cup \fF'_2 \cup \fG'_2)| + |\fD_{B1}| + \eta_\pls^{(3)} n\\
        &\geq |\fH_1| + |U_\pls^{(3)} \cap \bigcup \mathcal B^{(3)}_\pls| + \eta_\pls^{(3)} n\;,\\
    \end{align*}
    
    We set $\tilde U_\pls^{(3)} = \emptyset$ and immediately invoke the second part of proposition. We verify that for each $C \in \mathcal L$ we have
    \begin{align*}
        \avdeg(C)
        &\geq k + \delta k\\
        &=|\fD_A \cup \fD_B \cup W_A \cup W_B| + \delta k\\
        &\geq |\fH_1| + |\fH_2| + |\phi(W_A \cup W_B \cup \fF' \cup \fG') \cup U'| + \eta_\pls^{(3)} n\\
        &\geq |\fH_1| + |\fH_2| + |U_\pls^{(3)} \cup \tilde U_\pls^{(3)}| + \eta_\pls^{(3)} n\;.
    \end{align*}
    
    Thus we can extend $\phi$ to $\fH$. Moreover, note that after each application of Propositions~\ref{prop:embed-SM} and~\ref{prop:embed-LS} it was true that~$\phi$ avoided at least $r'\eta_\pls^{(1)}|C|/8$ vertices of each cluster~$C$. Thus, we can extend~$\phi$ to~$\fD_B$ as we promised in the first part of the analysis of this case. 

    \item
    We have defined $\phi$ on the whole tree~$T$ except for $\fF_1 \setminus \fF'_1$ and $\fG_1 \setminus \fG'_1$. We can again extend~$\phi$ to the whole~$T$ in the usual greedy manner. 
    
    \end{enumerate}

    %\newpage

\subsection*{Case D}
    In this case we assume the existence of two adjacent clusters $A, B$ such that $\avdeg(A,\mathcal  \fS_M \cup \mathcal L) \geq k + \delta k$ and $\avdeg(B,\mathcal  L) \geq \bone + \delta k$. Moreover, we assume that $\tr \aone \geq (1-\tr) \atwo$ and $\bone \leq \frac{\tr}{1-\tr} rk$ and for each edge $(C,D) \subseteq \mathbf M$ either $\avdeg(A, C) = 0$ or $\avdeg(A,D)=0$. 
    
    We proceed in the same way as in the previous case, although the analysis is different.

    \begin{figure}
        \centering
        \includegraphics{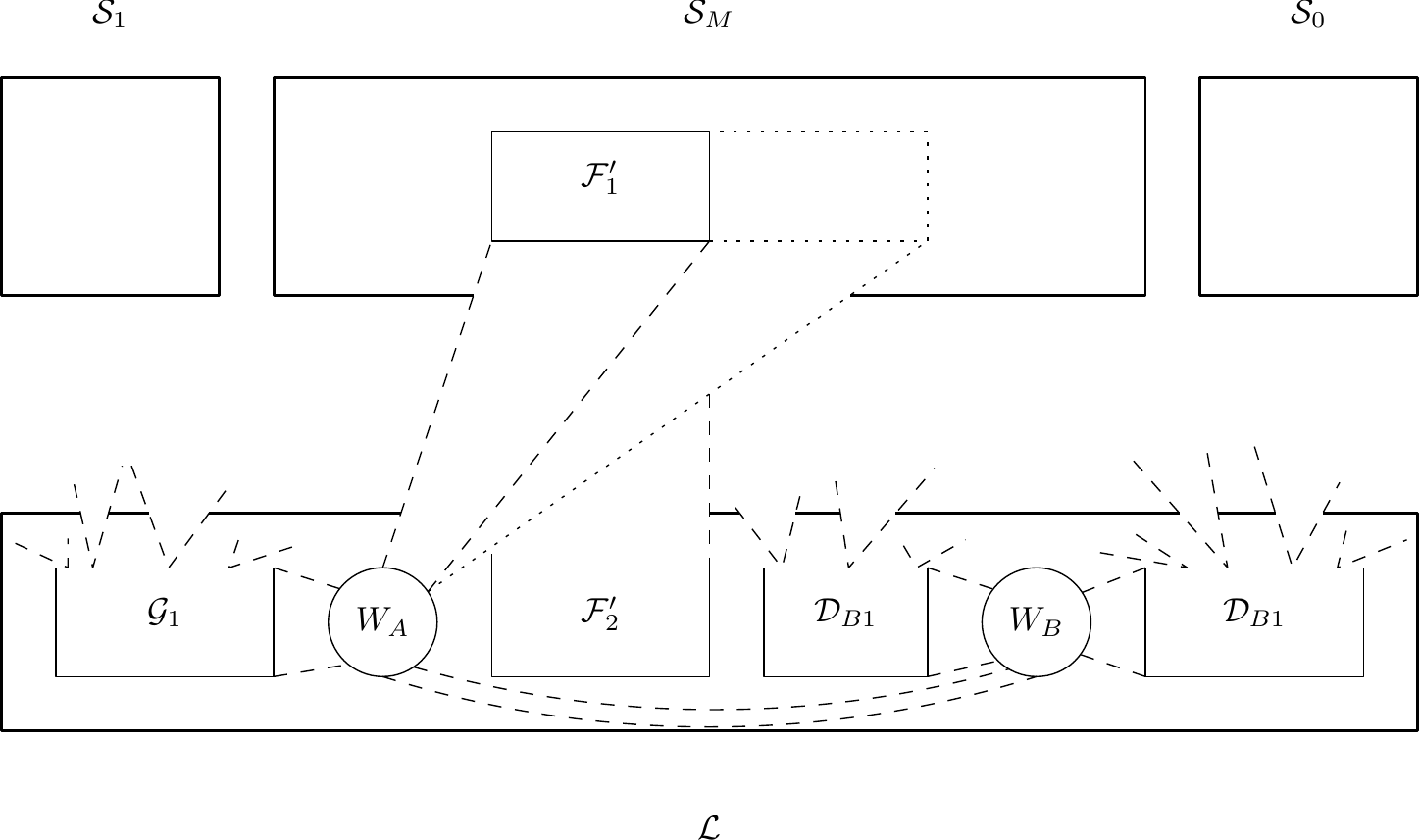}
        \caption{The embedding configuration in the case D. The order of operations is the same as in the preceding case, but the analysis is different. The figure suggests that as in the previous case we have to me more careful in the application of Proposition \ref{prop:embed-SM}. The special condition on the neighbourhood of the cluster $A$ plays the following role: we split the reserved vertices for $\fD_{B1}$ into two parts -- the vertices in the neighbourhood of $A$ (the right rectangle on the figure) and those that are not neighbours of $A$ (the left rectangle). Now the condition implies that the first type of vertices does not play a role in the embedding of $\fF'$ using the matching, whilst the second type of vertices does not have to be considered in the embedding of $\fG$ through the $\fL$-neighbourhood of $A$. }
        \label{fig:caseD}
    \end{figure}

    \begin{enumerate}
        \item 
        We start by reserving vertices for embedding the anchored forest~$F_\pls^{(1)}:=\fD_B = \fD_{B1} \cup \fD_{B2}$ such that $\fD_{B1}$ will be embedded in the $L$-neighbourhood of the cluster~$B$. This is done using Proposition~\ref{prop:embed-LS} in the exactly same way as in the previous case. 
        We get an embedding of $N(W_B) \cap \fD_B$ and a set of reserved vertices $W$. We set $U' = \phi(N(W_B) \cap \fD_B) \cup W$, $|U'| = |\fD_{B1}|$. We will also invoke the 'moreover' part Proposition~\ref{prop:embed-LS} after embedding the rest of~$T$ and then we set $\tilde U_\pls^{(1)} = \phi(\fF' \cup \fG)$ for $\fF' \cup \fG \subseteq \fD_A$. We have to ensure that for subsequent applications of Propositions~\ref{prop:embed-SM} and~\ref{prop:embed-LS} we have $\eta \ge \eta_\pls^{(1)} = q \delta /20$. 
        
        Moreover, we split the set $U' \subseteq \fL$ in two sets~$U'_1$ and~$U'_2$ such that~$U'_1$ contains the vertices from~$U'$ contained in clusters~$C$ such that $C \in N_{\mathbf H}(A)$ (we define $N_{\mathbf H}(A)$ as the set of clusters $C$ with $\avdeg(A, C) > 0$) and $U'_2 := U' \setminus U'_1$. Note that our assumption on the neighbourhood of cluster $A$ states that if we have $(C,D) \subseteq \mathbf M$ with $D \cap  U'_1\neq \emptyset$, we have then $\avdeg(A, C) = 0$. 
        
        \item
        We continue by embedding the anchored forest $\fD_A$ analogously to previous cases. Partition $\fD_A = \fF \cup \fG$, ordering the components by decreasing order f their skew, in such a way that $\fF$ is maximal with
        \begin{align}
            |\fF_2| \leq \frac{r'}{1-r'} \avdeg(A,\mathcal  \fS_M) - |U'_2| - \frac{r'}{1-r'} \delta k/2\;,
        \end{align}
        or $\fF$ is empty if the right hand side is smaller than zero. We define $\fF'$ as usual. If $\fF \not = \fD_A$, we have
        \begin{align}
            \label{D:F_large}
            |\fF_2| \geq \frac{r'}{1-r'} \avdeg(A,\mathcal  \fS_M) - |U'_2| - \frac{r'}{1-r'} \delta k/2 -  \beta k\;.
        \end{align}
        
        Moreover, $\fF$ is chosen so that it contains the trees with maximal skew, thus if it is non-empty we have
        \begin{align}
            \label{D:F_skew}
            \frac{|\fF_1|}{|\fF_2|} 
            \geq \frac{|\fF_1 \cup \fG_1|}{|\fF_2 \cup \fG_2|} 
            \geq \frac{1-\tr}{\tr}\;. 
        \end{align}

        Now we use Proposition~\ref{prop:embed-SM} to embed~$F_\psm:=\fF'$. Set $U_\psm := \phi(W_A \cup W_B) \cup U'_2$ and $\mathbf{M}_\psm$ be only those matching pairs $(C,D), C \subseteq \fS$ such that $\avdeg(A, C) > 0$.
        %\marginpar{tohle potrebuju ve zneni P19. D: alternativne muzes sikovne definovat $M$ jak jsi to popsal} 
        Observe that~$U'_1$ is disjoint from~$\bigcup V(\mathbf{M}_\psm)$. Set $\eta_\psm := \delta q / 3$, and $A_\psm:=A$. 
        As in the previous cases we easily verify that 
        \begin{align*}
            \avdeg(A,\mathcal \fS_M) 
            &\geq \frac{1-r'}{r'} |\fF_2| + \frac{1-r'}{r'} |U'_2|+ \delta k/2\\
            &\geq \frac{1-r'}{r'} |\fF_2| + \frac{1-r'}{r'} |U_\psm| + \delta k/3\\
            &\geq \frac{1-r'}{r'} |\fF'_2| + \sum_{C\subseteq  \fS\::\: CD\in \mathbf M}\max\{|U_\psm\cap C|, \frac{1-r'}{r'} |U_\psm\cap D|\} \; + \eta n.
        \end{align*}
        
        Thus we can extend $\phi$ to $\fF'$. Note that $\fF'_2$ is embedded in $L$-clusters that are not in the neighbourhood of $A$. Indeed, from our assumption on the cluster $A$ we have $\avdeg(A, D)=0$ for any edge $CD\in \mathbf{M}_\psm$, with $C\subseteq \fS$.
        %  , thus we will never extend $\phi$ to vertices from $U'_1$ (vertices such that the average degree of $A$ in the corresponding cluster is nonzero). Thus, there is no need in adding $U'_1$ to $U_\psm$. 
    
        \item
        We now apply Proposition~\ref{prop:embed-LS}, first part, to embed $F_\pls^{(2)}:=\fG$ if it is non-empty. Set $\mathcal B_\pls^{(2)} :=\mathcal  L \cap N_{\mathbf H}(A)$ and $U_\pls^{(2)} := \phi(W_A \cup W_B \cup \fF) \cup U'$. 
        Note that $U_\pls^{(2)} \cap \bigcup \mathcal B_\pls^{(2)} \subseteq \phi(W_A \cup W_B) \cup U'_1$, as we know that neither~$U'_2$, nor~$\phi(\fF'_2)$ is in~$N_{\mathbf H}(A)$ and $\phi(\fF'_1) \cap \fL = \emptyset$. Set $\eta_\pls^{(2)} := q \delta /4$, and $A_\pls^{(2)}:=A$. 
        
        We verify the first condition of the proposition:
        
        \begin{align*}
            \avdeg(A,\mathcal L) 
            &\geq k + \delta k - \avdeg(A, \fS_M)\\
            \just{bound \eqref{D:F_large}} & \geq k + \delta k - (\frac{1-r'}{r'} |\fF_2| + \frac{1-r'}{r'} |U'_2| + \delta k/2 + \frac{1-r'}{r'} \beta k)\\
            \just{definition of $U'$} & \geq k + \delta k - \frac{1-r'}{r'} |\fF_2| - \frac{1-r'}{r'} (|\fD_{B1}| - |U'_1|) - \delta k /2 - \frac{1-r'}{r'} \beta k\\
            \just{bounding error terms \& $\tilde r\le r'$} & \geq k - \frac{1-\tr}{\tr} |\fF_2| - \frac{1-\tr}{\tr} |\fD_{B1}| + \frac{1-r'}{r'} |U'_1| + \delta k /3\\
            \just{bound \eqref{D:F_skew} } &\geq k - |\fF_1| - \frac{1-\tr}{\tr} |\fD_{B1}| + |U'_1| + \delta k /3\\
            \just{assumed bound on $\fD_{B1}$} & \geq k - |\fF_1| - \frac{1-\tr}{\tr} \frac{\tr}{1-\tr} \tr k + |U'_1| + \delta k /3\\
            &= (1-\tr)k - |\fF_1| + |U'_1| + \delta k /3\\
            \just{bound on the skew of $T$} & \geq |\fD_{A1}| - |\fF_1| + |U'_1| + \delta k /3\\
            &\geq |\fG_1| + |U'_1| + \delta k /3\\
            &\geq |\fG_1| + |U_\pls^{(2)} \cap \bigcup \fB_\pls^{(2)}| + \eta_\pls^{(2)} n\;.\\
        \end{align*}
        
        We set $\tilde U_\pls^{(2)}:=\emptyset$ and immediately apply the second part of the proposition. We verify that for each $C \in \mathcal L$ we have
        
        \begin{align*}
            \avdeg(C)
            &\geq k + \delta k\\
            &=|\fD_A \cup \fD_{B} \cup W_A \cup W_B| + \delta k\\
            &\geq |\fG_1| + |\fG_2| + |\phi(W_A \cup W_B \cup \fF')| + |\fD_{B1}| + \eta_\pls^{(2)} n\\
            &\geq |\fG_1| + |\fG_2| + |U_\pls^{(2)} \cup \tilde U_\pls^{(2)}| + \eta_\pls^{(2)} n\;.
        \end{align*}
        
        Thus, we can extend $\phi$ to $\fG$. Moreover, after each operation it was true that~$\phi$ avoided at least $r'\eta_\pls^{(1)}|C|/8$ vertices of each cluster~$C$. Thus, we can extend~$\phi$ to~$\fD_B$ using the 'moreover' part of Proposition~\ref{prop:embed-LS}.
        \item
        We again extend the embedding of $T$ greedily to $\fF_1 \setminus \fF'_1$ as usual. 
    
    \end{enumerate}

\section{Conclusion}\label{sec:conclusion}

In this last section we show a straightforward application of our result and then consult the possibilities of further research in this area. 

\subsection*{Ramsey numbers for trees}
The Ramsey number $R(G_1, \ldots, G_m)$ is the least number such that any complete graph on $R(G_1, \ldots, G_m)$ vertices with its edges coloured with $m$ colours contains a monochromatic copy of $G_i$ in colour $i$ for some $1\leq i \leq m$. It is not difficult to see that, if true, both the Loebl--Koml\'os--S\'os conjecture and the Erd\H os--S\'os conjecture would imply that for any pair of trees $T_1, T_2$ on $k+1$ and $l+1$ vertices, respectively, it holds that $R(T_1, T_2) \leq k+l$. This was shown to be asymptotically true in~\cite{PS07+} and even finer asymptotic bound was obtained for $T_1=T_2$ in~\cite{Haxell2002}.

Our Conjecture~\ref{conj} generalises this consequence for trees of given skew. 

Suppose that we have trees $T_1, \ldots, T_m$ such that the size of the $i$-th tree is $k_i+1$ and the size of one of its colour class is at most $(k_i+1)/m$. Then, assuming the validity of Conjecture~\ref{conj}, we deduce $R(T_1, \ldots, T_m) \leq 2+\sum_{i=1}^{m} \left( k_i-1\right)$. Indeed, by the pigeonhole principle,  for every vertex $v$ there exists a colour $i$ such that $v$ is incident with at least $k_i$ edges of colour $i$. Moreover,  there exists a colour $c$ such that at least $1/m$ of the vertices are incident with at least $k_c$ edges of this colour. Thus, the subgraph formed by the edges of colour $c$ satisfies the conditions of Conjecture~\ref{conj}.
%This can be shown using Pigeonhole principle. 
%For each vertex it asserts that for some $i$ at least $k_i$ incident edges are coloured with the $i$-th colour. Moreover, for some fixed colour the number of vertices with many neighbours of that colour is at least $1/m$ times number of all vertices. Thus the complete graph restricted to edges with that particular colour fulfils the conditions of Conjecture~\ref{conj}. 
Using Theorem~\ref{thm:result}, we prove this consequence to be asymptotically true. 
%It is easy to see that it asserts existence of a suitable colour $i$ such that our graph restricted to edges coloured with $i$-th colour fulfils the conditions of Conjecture~\ref{conj} for embedding $T_i$. Using Theorem~\ref{thm:result}, we prove this to be asymptotically true. 

\begin{corollary}\label{cor:ramsey}
For trees $T_1, \ldots, T_m$ with $|T_i|=k_i$ and such that  one colour class of $T_i$ has size at most $k_i/m$ for $1\leq i \leq m$ we have 
$$R(T_1, \ldots, T_m) \leq \sum_{i=1}^{m} k_i + o\left(\sum_{i=1}^{m} k_i\right).$$% \textrm{ as } \sum_{i=1}^m k_i\rightarrow \infty.$$
\end{corollary}

This generalises the asymptotic bound from~\cite{PS07+} and can be shown in a very similar manner.  
%For each~$\delta$ we will prove that for $\sum_{i=1}^{m} k_i$ sufficiently large we have $R(T_1, \ldots, T_m) \leq \left(1+\delta\right)\sum_{i=1}^{m} k_i$. This can be done using Pigeonhole principle as above, but at first we will have to add few vertices to every tree to ensure it is big enough while preserving appropriate skew, so that we can use Theorem \ref{thm:result}. On the other hand we will choose the parameters~$\varepsilon$ and~$q$ small enough, so that the overall number of vertices will still be less than $\left(1+\delta\right)\sum_{i=1}^{m} k_i$. 

Note however that, if true, the Erd\H os--S\'os conjecture implies the same bound but without the additional restriction on the skew of the trees. 

\subsection*{Possible direction of research}
We believe that, similarly as in~\cite{Cooley2009,Hladkyn, Zhao2011}, one could use Simonovits' stability method to prove that Conjecture~\ref{conj} is true for dense graphs. Furthermore, by using techniques exposed in~\cite{HladkyLKS1,HladkyLKS2,HladkyLKS3,HladkyLKS4}, one can probably prove that Conjecture~\ref{conj} is asymptotically true even in the setting of sparse graphs.

\medskip

Considering the structure of the graph witnessing the tightness of Conjecture~\ref{conj} given in Section~\ref{sec:intro}, it might seem feasible to strengthen the conjecture by replacing the condition on the size of the smaller colour class by the same condition on the size of the complement of a maximal independent set. However, this is not possible; a complete bipartite graph $K_{(k-1)/2,k}$ does not contain a bistar~$B_{(k-1)/2,(k-1)/2}$ (that is, two stars with $(k-1)/2$ leaves with their centres joined by an edge) for $k\geq 7$ odd, even though almost $1/3$ of vertices of $K_{(k-1)/2,k}$ have degree at least~$k$ and the size of the complement of a maximal independent set in $B_{(k-1)/2,(k-1)/2}$ is $2$, i.e., its relative size with respect to the whole bistar is very small, in particular at most $1/4$.

\bibliographystyle{alpha}
\bibliography{bibl}

\end{document}